\documentclass[12pt]{article}
\usepackage[english]{babel}
\usepackage{amsmath}
\usepackage{amsfonts}
\usepackage{amssymb}
\usepackage{eucal}
\usepackage{amsthm}
\usepackage{latexsym}
\usepackage{epic}
\numberwithin{equation}{section}
\usepackage{graphicx,graphics}
\numberwithin{equation}{section}
\usepackage{graphicx,graphics}
\usepackage{pdfsync}
\usepackage{epstopdf}
\def \dis {\displaystyle}
\def \tex {\textstyle}

\def \limo {\;\;\mathop{\longrightarrow}_{\varepsilon\to 0}\;\;}
\def \MMMMI#1{-\mskip -#1mu\int}
\def \moy {\displaystyle\MMMMI{19.2}}

\def \confai {\;-\kern -.5em\rightharpoonup\,}
\def \cqfd {\hfill$\Box$}
\def \div{\mbox{\rm div}}
\def \Div{\mbox{\rm Div}}
\def \curl{\mbox{\rm curl}}

\def \al {\alpha}

\def \ga {\gamma}
\def \Ga {\Gamma}
\def \de {\delta}
\def \De {\Delta}
\def \ep {\varepsilon}
\def \om {\omega}
\def \Om {\Omega}
\def \la {\lambda}

\def \ph {\varphi}
\def \th {\theta}
\def \ka {\kappa}
\def \si {\sigma}
\def \Si {\Sigma}

\def \ZZ {\mathbb Z}
\def \RR {\mathbb R}

\def \D {\mathcal{D}}

\def \M {\mathcal{M}}

\def \bmu {\mbox{\boldmath $\mu$}}
\def \beq {\begin{equation}}
\def \eeq {\end{equation}}
\def \ba {\begin{array}}
\def \ea {\end{array}}
\def \bs {\bigskip}
\def \ms {\medskip}
\def \ss {\smallskip}
\def \ecart {\noalign{\medskip}}
\newtheorem{Thm}{Theorem}[section]

\newtheorem{Pro}[Thm]{Proposition}
\newtheorem{Lem}[Thm]{Lemma}

\newtheorem{Adef}[Thm]{Definition}

\newtheorem{Arem}[Thm]{Remark}
\newenvironment{Rem}{\begin{Arem}\rm}{\end{Arem}}
\newtheorem{Aexa}[Thm]{Example}

\newtheorem{Anot}[Thm]{Notation}

\def \refe #1.{(\ref{#1})}
\def \reff #1.{figure~\ref{#1}}
\def \refs #1.{Section~\ref{#1}}
\def \refss #1.{Subsection~\ref{#1}}
\def \refD #1.{Definition~\ref{#1}}
\def \refT #1.{Theorem~\ref{#1}}
\def \refL #1.{Lemma~\ref{#1}}
\def \refC #1.{Corollary~\ref{#1}}
\def \refP #1.{Proposition~\ref{#1}}
\def \refPt #1.{Properties~\ref{#1}}
\def \refR #1.{Remark~\ref{#1}}
\def \refE #1.{Example~\ref{#1}}
\def \refN #1.{Notation~\ref{#1}}
\def \Wp {W_{\sharp,\ep}^\la}

\def \wp {w_{\sharp,\ep}^\la}
\def \Qp {Q_{\sharp,\ep}^\la}
\def \qp {q_{\sharp,\ep}^\la}
\author{
\footnotesize
\begin{tabular}{cccc}
\normalsize Marc BRIANE && \normalsize Patrick G\'ERARD
\\
INSA de Rennes && Universit\'e Paris-Sud
\\
IRMAR, CNRS UMR 6625 && LMO, CNRS UMR 8628
\\
35708 Rennes Cedex 7 FRANCE && 91405 Orsay FRANCE
\\
&& Institut Universitaire de France
\\
mbriane@insa-rennes.fr && Patrick.Gerard@math.u-psud.fr
\end{tabular}
}
\title{A drift homogenization problem revisited}
\begin{document}
\maketitle
\begin{abstract}
This paper revisits a homogenization problem studied by L.~Tartar related to a tridimensional Stokes equation perturbed by a drift (connected to the Coriolis force). Here, a scalar equation and a two-dimensional Stokes equation with a $L^2$-bounded oscillating drift are considered. Under higher integrability conditions the Tartar approach based on the oscillations test functions method applies and leads to a limit equation with an extra zero-order term. When the drift is only assumed to be equi-integrable in $L^2$, the same limit behaviour is obtained. However, the lack of integrability makes difficult the direct use of the Tartar method. A new method in the context of homogenization theory is proposed. It is based on a parametrix of the Laplace operator which permits to write the solution of the equation as a solution of a fixed point problem, and to use truncated functions even in the vector-valued case. On the other hand, two counter-examples which induce different homogenized zero-order terms actually show the optimality of the equi-integrability assumption.
\end{abstract}
{\bf Keywords:} Homogenization - Second-order elliptic equations - Drift
\par\bs\noindent
{\bf Mathematics Subject Classification:}  35B27 - 76M50
\section{Introduction}
At the end of the Seventies L.~Tartar developed his method based on oscillating test functions to deal with the homogenization of PDE's. In the particular framework of hydrodynamics \cite{Tar1,Tar2} he studied the Stokes equation in a bounded domain $\Om$ of $\RR^3$, perturbed by an oscillating drift term, i.e.
\beq\label{Sto}
\left\{\ba{rl}
-\,\De u_\ep+\curl\left(v_\ep\right)\times u_\ep+\nabla p_\ep=f & \mbox{in }\Om
\\ \ecart
\div\left(u_\ep\right)=0 & \mbox{in }\Om
\\ \ecart
u_\ep=0 & \mbox{on }\Om,
\ea\right.
\eeq
where the oscillations are produced by the sequence of vector-valued functions $v_\ep$ which weakly converges to some $v$ in $L^3(\Om)^3$. L.~Tartar proved that the limit equation of \refe{Sto}. is the Brinkman \cite{Bri} type equation
\beq\label{hSto}
\left\{\ba{rll}
-\,\De u+\curl\left(v\right)\times u+\nabla p+{M}u=f & \mbox{in }\Om
\\ \ecart
\div\left(u\right)=0 & \mbox{in }\Om
\\ \ecart
u=0 & \mbox{on }\Om,
\ea\right.
\eeq
where ${M}$ is a positive definite symmetric matrix-valued function. More precisely, ${M}$ is defined by the convergences
\beq\label{MTi}
(Dw_\ep^\la)^T\,v_\ep\confai {M}\la\quad\mbox{weakly in }L^{3\over 2}(\Om)^3,\quad\mbox{for any }\la\in\RR^3,
\eeq
where $w_\ep^\la\in W^{1,3}(\Om)^3$ solves the Stokes equation \refe{Sto}. in which $u_\ep$ is replaced by $\la$. Then, the convergence \refe{MTi}. combined with the compactness of $u_\ep$ in $L^3(\Om)^3$, yields the zero-order term $Mu$ in \refe{hSto}.. In  \cite{Tar3} L.~Tartar revisited this problem using the H-measures tool. On the other hand, the appearance of such a strange zero-order term in homogenization was also obtained from finely perforated domains by D.~Cioranescu, F.~Murat \cite{CiMu} for the Laplace equation, and by  G.~Allaire~\cite{All} for the Stokes equation, with zero Dirichlet boundary condition on the holes.
\par
Since $\curl\left(v_\ep\right)\times u_\ep$ is orthogonal to $u_\ep$, the energy associated with \refe{Sto}. is reduced to
\beq\label{ener}
\int_\Om|Du_\ep|^2\,dx,
\eeq
and thus does not depend on the drift $v_\ep$. Starting from this remark our aim is to study two drift homogenization problems associated with the same energy \refe{ener}., and to specify the optimal integrability satisfied by the drift so that the Tartar approach holds. The first problem is scalar and the second problem is a two-dimensional equivalent of the Stokes problem \refe{Sto}.. However, we have not succeeded in obtaining an optimal result for the three-dimensional Stokes equation \refe{Sto}. since the best integrability assumption for $v_\ep$ is not clear.
\par\bs
In Section~\ref{s.sca}, we consider the following scalar equation in a bounded open set $\Om$ of $\RR^N$,
\beq\label{edrift}
\left\{\ba{rll}
-\,\De u_\ep+b_\ep\cdot\nabla u_\ep+{\rm div}\left(b_\ep\,u_\ep\right) & =f & \mbox{in }\Om
\\ \ecart
u_\ep & =0 & \mbox{on }\partial\Om,
\ea\right.
\eeq
where $b_\ep\in L^\infty(\Om)^N$ is bounded in $L^2(\Om)^N$. We obtain three different homogenization results:
\par
In Section~\ref{ss.T}, assuming that the divergence of the drift $b_\ep$ is bounded in $W^{-1,q}(\Om)$, with $q>N$, we prove (see Theorem~\ref{thmT1}) that the sequence $u_\ep$ weakly converges in $H^1_0(\Om)$ to the solution $u$ of the equation
\beq\label{hdrift}
\left\{\ba{rll}
-\,\De u+b\cdot\nabla u+{\rm div}\left(b\,u\right)+\mu\,u & =f & \mbox{in }\Om
\\ \ecart
u & =0 & \mbox{on }\partial\Om,
\ea\right.
\eeq
where $\mu$ is a nonnegative function. The proof follows the Tartar method using the oscillating test function
\beq\label{wei}
w_\ep:=\De^{-1}\bigl(\div\left(b_\ep\right)\bigr)\in H^1_0(\Om).
\eeq
\par
Then, in Section~\ref{ss.equi}, assuming only the equi-integrability of the sequence $\nabla w_\ep$ in $L^2(\Om)^N$ (this is actually a weaker assumption than the equi-integrability of the whole sequence $b_\ep$), we obtain (see Theorem~\ref{thmT2}) the limit problem \refe{hdrift}. with
\beq\label{mui}
|\nabla w_\ep-\nabla w|^2\confai\mu\;\;\mbox{weakly in }L^1(\Om)\quad\mbox{and}\quad\mu\,u^2\in L^1(\Om).
\eeq
It seems intricate to apply directly the Tartar method with the test function $w_\ep$, since we cannot control the terms $b_\ep\cdot\nabla u_\ep\,w_\ep$ and $b_\ep\cdot\nabla w_\ep\,u_\ep$. To this end, one should consider truncations of both $w_\ep$ and $\nabla w_\ep$. To overcome this difficulty we propose a new method, up to our knowledge, in the context of homogenization theory, based on a parametrix of the Laplace operator. It follows that $u_\ep$ reads as a solution of a fixed point problem, which allows us to estimate the sequence $\nabla w_\ep\cdot\nabla u_\ep$ only using a truncation of $\nabla w_\ep$. The equi-integrability of $\nabla w_\ep$ then gives the thesis.  
Also assuming that $b\in L^q(\Om)^N$, with $q>N$, (which ensures the uniqueness in \refe{hdrift}.) we prove the following corrector result
\beq\label{corresi}
u_\ep-\left(1+w_\ep-w\right)u\;\longrightarrow\;0\quad\mbox{strongly in }W^{1,q}_{\rm loc}(\Om),\quad\mbox{for any }q\in[1,N').
\eeq 
\par
Finally, in Section~\ref{s.cex}, we show the optimality of the equi-integrability condition thanks to a counter-example in the periodic framework (see Theorem~\ref{thm.cex}). Making a change of functions with $b_\ep=\nabla w_\ep$, equation \refe{edrift}. is shown to be equivalent to the following equation
\beq\label{ezot}
-\,\De v_\ep+\mu_\ep\,v_\ep=f_\ep,\quad\mbox{with}\quad\mu_\ep:=|\nabla w_\ep|^2,
\eeq
the solution of which has the same limit as $u_\ep$. G.~Dal Maso, A.~Garroni \cite{DMGa} proved that the class of equations of type \refe{ezot}. is stable under homogenization. Here, we do not use this general result, but we explicit an oscillating sequence $w_\ep$ so that the limit equation of \refe{edrift}., or equivalently \refe{ezot}., is
\beq\label{hcex}
-\,\De u+\ga\,u=f,
\eeq
with an explicit constant $\ga$ which turns out to be $<\mu$.
Therefore, the loss of equi-integrability for $\nabla w_\ep$ violates the result of Section~\ref{ss.equi}. Note that the vectorial character of the drift term in equation \refe{edrift}. makes difficult the derivation of a closure result similar to the one of \cite{DMGa} which is strongly based on a comparison principle.
\par\bs
In Section~\ref{s.Sto}, we consider the following two-dimensional equivalent of the perturbed Stokes problem \refe{Sto}.,
\beq\label{Sto2d}
\left\{\ba{rl}
-\,\De u_\ep+\curl\left(v_\ep\right) Ju_\ep+\nabla p_\ep=f & \mbox{in }\Om
\\ \ecart
\div\left(u_\ep\right)=0 & \mbox{in }\Om
\\ \ecart
u_\ep=0 & \mbox{on }\Om,
\ea\right.
\eeq
where $J$ is the rotation matrix of angle $90^\circ$, and $v_\ep\in L^\infty(\Om)^2$ is bounded in $L^2(\Om)^2$. We follow the same scheme as in the scalar case:
\par
In Section~\ref{ss.TS}, assuming that the sequence $v_\ep$ is bounded in $L^r(\Om)^2$ with $r>2$, we show (see Theorem~\ref{thmT2}) that the sequence $u_\ep$ weakly converges in $H^1_0(\Om)$ to the solution $u$ of the Brinkman equation
\beq\label{hSto2d}
\left\{\ba{rl}
-\,\De u+\curl\left(v\right) Ju+\nabla p+{M}u=f & \mbox{in }\Om
\\ \ecart
\div\left(u\right)=0 & \mbox{in }\Om
\\ \ecart
u=0 & \mbox{on }\Om,
\ea\right.
\eeq
where ${M}$ is a symmetric positive definite matrix-valued function defined by the convergence \refe{MTi}. in $L^{2r\over r+2}(\Om)^2$.
\par 
In Section~\ref{ss.equiS}, assuming only the equi-integrability of the sequence $v_\ep$ in $L^2(\Om)^2$, we prove (see Theorem~\ref{thm2}) owing to the Tartar method that the sequence $u_\ep$ weakly converges in $H^1_0(\Om)$ to the solution $u$ of the Brinkman equation \refe{hSto2d}. with similarly to \refe{mui}.,
\beq\label{MTi2d}
(Dw_\ep^\la)^T\,v_\ep\confai {M}\la\quad\mbox{weakly in }L^1(\Om)^2\quad\mbox{and}\quad {M}u\cdot u\in L^1(\Om).
\eeq
The proof is based on a double parametrix method carrying on both the velocity $u_\ep$ and the pressure $p_\ep$. However, the proof of the last estimate of \refe{MTi2d}. is more delicate than the one of~\refe{mui}., since we cannot use a comparison principle as in the scalar case. We need to introduce a test function similar to $w_\ep^\la$ but associated with a truncation of $v_\ep$. Moreover, if $\Om$ has a regular boundary, $v\in L^r(\Om)^2$ with $r>2$, and ${M}\in L^m(\Om)^{2\times 2}$ with $m>1$, we get the corrector result
\beq\label{corresiS}
u_\ep-u-W_\ep\,u\;\longrightarrow\;0\quad\mbox{strongly in }W^{1,1}(\Om)^2,\quad\mbox{where}\quad W_\ep\la:=w_\ep^\la,\mbox{ for }\la\in\RR^2.
\eeq
\par
Finally, in Section~\ref{ss.equiS}, we construct an oscillating sequence $v_\ep$ which is not equi-integrable in $L^2(\Om)^2$, which leads to the limit problem \refe{hSto2d}. involving a matrix $\Ga$ which is not symmetric and satisfies the strict inequality
\[
\Ga\la\cdot\la<{M}\la\cdot\la,\quad\mbox{for any }\la\neq 0,
\]
which is inconsistent with the Tartar approach. This shows the optimality of the equi-integrability condition as in the scalar case. It would be very interesting to find the closure of the family of problems \refe{Sto2d}. under the sole condition of $L^2$-boundedness of the sequences $v_\ep$. This problem is far from being evident due to the absence of comparison principle for such a vector-valued equation.    
\subsection*{Notations}
\begin{itemize}
\item The space dimension is $N\geq 2$, and $\dis 2^*:={2N\over N-2}$.
\item The conjugate exponent of $p\geq 1$ is denoted by $\dis p':={p\over p-1}$.
\item For $u:\RR^N\longrightarrow\RR^N$, $\dis Du:=\left({\partial u_{i}\over\partial x_j}\right)_{1\leq i,j\leq N}$.
\item For $\Si:\RR^N\longrightarrow\RR^{N\times N}$, $\dis\Div\left(\Si\right):=\left(\sum_{j=1}^N{\partial\Si_{ij}\over\partial x_j}\right)_{1\leq i\leq N}$. 
\item $H^1_\sharp(Y)$, with $Y:=(0,1)^N$, denotes the space of the $Y$-periodic functions on $\RR^N$ which belong to $H^1_{\rm loc}(\RR^N)$.
\end{itemize}
\section{A scalar equation with a drift term}\label{s.sca}
Along this section $\Om$ is a bounded regular open set of $\RR^N$, with $N\geq 2$, and $f$ is a distribution in $H^{-1}(\Om)$.
\subsection{The classical case}\label{ss.T}
Let $q\in(N,\infty)$. Consider a sequence $b_\ep$ in $L^\infty(\Om)^N$ such that 
\beq\label{be}
b_\ep\confai b\quad\mbox{weakly in }L^2(\Om)^N\qquad\mbox{and}\qquad{\rm div}\left(b_\ep\right)\mbox{ is bounded in }W^{-1,q}(\Om).
\eeq
Let $w_\ep\in W^{1,q}_0(\Om)$ be the solution of the equation (see, e.g., \cite{Gey} Theorem~2.1)
\beq\label{we}
\De w_\ep={\rm div}\left(b_\ep\right)\quad\mbox{in }\D'(\Om).
\eeq
Up to a subsequence $w_\ep$ weakly converges in $W^{1,q}_0(\Om)$ to the function $w$ solution of
\beq\label{w}
\De w={\rm div}\left(b\right)\quad\mbox{in }\D'(\Om).
\eeq
We have the following result:
\begin{Thm}\label{thmT1}
The solution $u_\ep\in H^1_0(\Om)$ of the equation
\beq\label{ue}
-\,\De u_\ep+b_\ep\cdot\nabla u_\ep+{\rm div}\left(b_\ep\,u_\ep\right)=f\quad\mbox{in }\D'(\Om),
\eeq
weakly converges in $H^1_0(\Om)$, up to a subsequence, to a solution $u\in H^1_0(\Om)$ of the equation
\beq\label{eq1u}
-\,\De u+b\cdot\nabla u+{\rm div}\left(b\,u\right)+\mu\,u=f\quad\mbox{in }\D'(\Om),
\eeq
where $\mu$ is the function defined by the convergence
\beq\label{mu}
|\nabla w_\ep-\nabla w|^2\confai\mu\quad\mbox{weakly in }L^{q\over 2}(\Om).
\eeq
\end{Thm}
\begin{Rem}
The uniqueness for equation \refe{ue}. is not evident under the sole assumption $b\in L^2(\Om)^2$. Assuming a stronger integrability of $b$ we will obtain in Theorem~\ref{thm1} the uniqueness for the limit equation.
\end{Rem}
\begin{proof}
The proof is based on the Tartar method of the oscillating test functions (see Appendix of~\cite{San}, and \cite{Tar4}). The function $w_\ep$ of \refe{we}. will play the role of the oscillating test function.
The variational formulation of \refe{ue}. is
\beq\label{vfue}
\int_\Om\nabla u_\ep\cdot\nabla\ph\,dx+\int_\Om b_\ep\cdot\nabla u_\ep\,\ph\,dx-\int_\Om b_\ep\cdot\nabla\ph\,u_\ep\,dx
=\langle f,\ph\rangle_{H^{-1}(\Om),H^1_0(\Om)}\,,\quad\forall\,\ph\in H^1_0(\Om).
\eeq
Then, by the Lax-Milgram theorem there exists a unique solution $u_\ep$ of \refe{vfue}. in $H^1_0(\Om)$.
In particular, for $v\in W^{1,\infty}(\Om)$, putting $\ph=v\,u_\ep$ as test function in \refe{vfue}. we obtain the identity
\beq\label{vue}
\int_\Om|\nabla u_\ep|^2\,v\,dx+\int_\Om\nabla u_\ep\cdot\nabla v\,u_\ep\,dx-\int_\Om b_\ep\cdot\nabla v\,u_\ep^2\,dx
=\langle f,v\,u_\ep\rangle_{H^{-1}(\Om),H^1_0(\Om)}\,,
\eeq
which will be used several times. So, choosing $v=1$ in \refe{vue}. the term with $b_\ep$ cancel so that we easily deduce that $u_\ep$ is bounded in $H^1_0(\Om)$ and weakly converges, up to a subsequence, to a function $u$ in $H^1_0(\Om)$.
Therefore, it follows from \refe{vfue}.  the limit variational formulation
\beq\label{vfe1}
\int_\Om\nabla u\cdot\nabla\ph\,dx+\int_\Om b\cdot\nabla u\,\ph\,dx+\int_\Om\ph\,d\nu-\int_\Om b\cdot\nabla\ph\,u\,dx
=\langle f,\ph\rangle_{H^{-1}(\Om),H^1_0(\Om)}\,,
\eeq
which holds for any $\ph\in W^{1,q}_0(\Om)$ (due to the embedding of $W^{1,q}_0(\Om)$ into $C(\bar\Om)$ for $q>N$), where the measure $\nu$ is defined by the convergence
\beq\label{nu}
b_\ep\cdot\nabla u_\ep\confai b\cdot\nabla u+\nu\quad\mbox{weakly-$*$ in }\M(\Om).
\eeq
The limit equation associated with \refe{vfe1}. is
\beq\label{e1}
-\,\De u+b\cdot\nabla u+\nu+{\rm div}\left(b\,u\right)=f\quad\mbox{in }\D'(\Om).
\eeq
\par
Now, let us determine the measure $\nu$ of \refe{nu}.. Let $\ph\in C^\infty_c(\Om)$. Putting $\ph\,w_\ep$ as test function in \refe{vfue}. and $\ph\,u_\ep$ in \refe{we}., and taking the difference of the two equalities we get
\beq\label{e2}
\ba{l}
\dis \int_\Om\nabla u_\ep\cdot\nabla\ph\,w_\ep\,dx-\int_\Om\nabla w_\ep\cdot\nabla\ph\,u_\ep\,dx
\\ \ecart
\dis =\langle f,\ph\,w_\ep\rangle_{H^{-1}(\Om),H^1_0(\Om)}-\int_\Om b_\ep\cdot\nabla u_\ep\,\ph\,w_\ep\,dx+\int_\Om b_\ep\cdot\nabla w_\ep\,\ph\,u_\ep\,dx
+\int_\Om b_\ep\cdot\nabla\ph\,u_\ep\,w_\ep\,dx
\\ \ecart
\dis -\int_\Om b_\ep\cdot\nabla u_\ep\,\ph\,dx-\int_\Om b_\ep\cdot\nabla\ph\,u_\ep\,dx.
\ea
\eeq
Passing to the limit in \refe{e2}. by using the strong convergence of $u_\ep$ in $L^p(\Om)$, for $p<2^*$, and the uniform convergence of $w_\ep$ in $C(\bar\Om)$ ($q>N$), we obtain
\beq\label{e3}
\ba{l}
\dis \int_\Om\nabla u\cdot\nabla\ph\,w\,dx-\int_\Om\nabla w\cdot\nabla\ph\,u\,dx
\\ \ecart
\dis =\langle f,\ph\,w\rangle_{H^{-1}(\Om),H^1_0(\Om)}-\int_\Om b\cdot\nabla u\,\ph\,w\,dx-\int_\Om\ph\,w\,d\nu+\int_\Om\si\,\ph\,u\,dx+\int_\Om b\cdot\nabla\ph\,u\,w\,dx
\\ \ecart
\dis -\int_\Om b\cdot\nabla u\,\ph\,dx-\int_\Om\ph\,d\nu-\int_\Om b\cdot\nabla\ph\,u\,dx,
\ea
\eeq
where the measure $\nu$ is defined by \refe{nu}. and the function $\si$ is defined, up to a subsequence, by the convergence
\beq\label{si}
b_\ep\cdot\nabla w_\ep\confai\si\quad\mbox{weakly in }L^{2q\over q+2}(\Om).
\eeq
On the other hand, putting $\ph\,w\in W^{1,q}_0(\Om)$ in \refe{vfe1}. and $\ph\,u\in H^1_0(\Om)$ in \refe{w}. we have
\beq\label{unu}
\ba{ll}
\dis \int_\Om\nabla u\cdot\nabla(\ph\,w)\,dx & \dis =\langle f,\ph\,w\rangle_{H^{-1}(\Om),H^1_0(\Om)}-\int_\Om b\cdot\nabla u\,\ph\,w\,dx-\int_\Om\ph\,w\,d\nu
\\ \ecart
& \dis +\int_\Om b\cdot\nabla w\,\ph\,u\,dx +\int_\Om b\cdot\nabla\ph\,u\,w\,dx,
\ea
\eeq
\beq\label{wsi}
\int_\Om\nabla w\cdot\nabla(\ph\,u)\,dx=\int_\Om b\cdot\nabla u\,\ph\,dx+\int_\Om b\cdot\nabla\ph\,u\,dx.
\eeq
Equating the difference between \refe{unu}. and \refe{wsi}. to the right-hand side of \refe{e3}., it follows that
\beq\label{e4}
\int_\Om\si\,\ph\,u\,dx-\int_\Om b\cdot\nabla w\,\ph\,u\,dx-\int_\Om\ph\,d\nu=0,\quad\mbox{for any }\ph\in C^\infty_c(\Om),
\eeq
which implies that
\beq\label{nusi}
\nu=\si\,u-b\cdot\nabla w\,u\quad\mbox{in }\D'(\Om).
\eeq
It thus remains to determine the limit equation \refe{eq1u}.. To this end, we pass to the limit by using $\ph\,w_\ep$ as test function in \refe{we}. and the definition \refe{mu}. of $\mu$, and we put $\ph\,w$ in \refe{w}., which yields
\beq\label{e5}
\int_\Om\left(\mu+|\nabla w|^2\right)\ph\,dx+\int_\Om\nabla w\cdot\nabla\ph\,w\,dx=\int_\Om\si\,\ph\,dx+\int_\Om b\cdot\nabla\ph\,w\,dx,
\eeq
\beq\label{e6}
\int_\Om|\nabla w|^2\,\ph\,dx+\int_\Om\nabla w\cdot\nabla\ph\,w\,dx=\int_\Om b\cdot\nabla w\,\ph\,dx+\int_\Om b\cdot\nabla\ph\,w\,dx.
\eeq
Equating \refe{e5}. and \refe{e6}., we deduce that
\beq\label{musi}
\mu=\si-b\cdot\nabla w\quad\mbox{in }\D'(\Om),
\eeq
which combined with \refe{nusi}. implies that
\beq\label{nu=}
\nu=\mu\,u\quad\mbox{in }\D'(\Om).
\eeq
Finally, the limit equation \refe{e1}. and the relation \refe{nu=}. give the desired homogenized equation~\refe{eq1u}..
\end{proof}
\begin{Rem}
It can be shown that
\beq\label{muH}
\mu(x)=\int_{S^{N-1}}\mbox{\bmu}\left(x,d\xi\right)\xi\cdot\xi\,,
\eeq
where {\bmu} denotes the matrix-valued $H$-measure (or micro-local defect measure) of the sequence $b_\ep$ (see \cite{Tar3} and \cite{Ger}), and $S^{N-1}$ the unit sphere of $\RR^N$.
\end{Rem}
Assumption \refe{be}. is actually not sharp. In the next section we replace it by the boundedness of $b_\ep$ and the equi-integrability of $\nabla w_\ep$ in $L^2(\Om)^2$. 
\subsection{The case under an equi-integrability assumption}\label{ss.equi}
\noindent
In this section $\Om$ is a bounded open set of $\RR^N$.
Consider a sequence $b_\ep$ in $L^\infty(\Om)^N$ the Hodge decomposition of which is
\beq\label{Hbe}
b_\ep=\nabla w_\ep+\xi_\ep,\quad\mbox{with}\quad w_\ep\in H^1_0(\Om),\ \xi_\ep\in L^2(\Om)^N\mbox{ and }\div\left(\xi_\ep\right)=0,
\eeq
such that
\beq\label{beb}
b_\ep\confai b\quad\mbox{weakly in }L^2(\Om)^N.
\eeq
Note that for a fixed $\ep>0$, $w_\ep\in W^{1,p}(\Om)$ and $\xi_\ep\in L^p(\Om)^N$ for any $p\in[2,\infty)$. But the essential point is the asymptotic behaviour of the sequences $b_\ep$, $\nabla w_\ep$, $\xi_\ep$.
Our main assumption is the equi-integrability of the sequence $\nabla w_\ep$ in $L^2(\Om)^N$.  By virtue of the Vitali-Saks theorem this is equivalent to the following convergence, up to an extraction of a subsequence,
\beq\label{equiDwe}
|\nabla w_\ep-\nabla w|^2\confai\mu\quad\mbox{weakly in }L^1(\Om),
\eeq
(Compare to \refe{mu}. with $q>N$). 
\par
We have the following result:
\begin{Thm}\label{thm1}
\hfill
\par\ss\noindent
$i)$ Under the equi-integrability assumption \refe{equiDwe}. the solution $u_\ep$ of \refe{ue}. weakly converges in $H^1_0(\Om)$ to a solution $u$ of the equation
\beq\label{eq2u}
-\,\De u+b\cdot\nabla u+{\rm div}\left(b\,u\right)+\mu\,u=f\quad\mbox{in }\D'(\Om),
\eeq
with 
\beq\label{muu}
\int_\Om\mu\,u^2\,dx\leq\langle f,u\rangle_{H^{-1}(\Om),H^1_0(\Om)}-\int_\Om|\nabla u|^2\,dx.
\eeq
\par\noindent
$ii)$ Also assume that $b\in L^q(\Om)^N$, where $q>2$ if $N=2$ and $q=N$ if $N>2$. Then, we have
\beq\label{muu=}
\int_\Om|\nabla u|^2\,dx+\int_\Om\mu\,u^2\,dx=\langle f,u\rangle_{H^{-1}(\Om),H^1_0(\Om)}.
\eeq
and there exists a unique solution $u\in H^1_0(\Om)$ of equation \refe{eq2u}., with $\mu\,u^2\in L^1(\Om)$.
\\
Moreover, for any $p\in[1,2)$ if $N=2$ and $p=N'$ if $N>2$, we have the corrector result
\beq\label{corDue}
\nabla u_\ep-\nabla u-\left(\nabla w_\ep-\nabla w\right)u\;\longrightarrow\;0\quad\mbox{strongly in }L^p_{\rm loc}(\Om)^N,
\eeq
and for any $r\in[1,p)$, 
\beq\label{corue}
u_\ep-\left(1+w_\ep-w\right)u\;\longrightarrow\;0\quad\mbox{strongly in }W^{1,r}_{\rm loc}(\Om).
\eeq
\end{Thm}
\begin{Rem}
No equi-integrability is required for the divergence free sequence $\xi_\ep$. Actually, we can prove that the equi-integrability of the sequence $b_\ep$ in $L^2(\Om)^N$ implies the equi-integrability of its two components $\nabla w_\ep$, $\xi_\ep$ in $L^2_{\rmÊloc}(\Om)^N$. Therefore, condition \refe{equiDwe}. is really weaker than the equi-integrability of $b_\ep$.
\par
Moreover, the equi-integrability of $\nabla w_\ep$ in $L^2(\Om)^N$ is essential for deriving the limit equation with the zero-order term $\mu\,u$. When this condition is not satisfied we can obtain a similar limit equation but with a different zero-order term (see Section~\ref{s.cex}).
\end{Rem}
\noindent
{\bf Proof of Theorem~\ref{thm1}.}
The limit $u$ of $u_\ep$ in $H^1_0(\Om)$ solves the equation \refe{e1}. where $\nu$ is defined by
\beq\label{nu1}
b_\ep\cdot\nabla u_\ep-b\cdot\nabla u\confai\nu\quad\mbox{weakly-$*$ in }\M(\Om),
\eeq
By the Murat, Tartar div-curl lemma \cite{Mur} we have
\[
b_\ep\cdot\nabla u_\ep=\left(\xi_\ep+\nabla w\right)\cdot\nabla u_\ep+\left(\nabla w_\ep-\nabla w\right)\cdot\nabla u_\ep
\confai b\cdot\nabla u+\nu\quad\mbox{in }\D'(\Om).
\]
This combined with the equi-integrability of $\nabla w_\ep$ implies that $\nu$ is also given by the convergence
\beq\label{nu2}
\left(\nabla w_\ep-\nabla w\right)\cdot\nabla u_\ep\confai\nu\quad\mbox{weakly in }L^1(\Om).
\eeq
\par
The proof of Theorem~\ref{thm1}. is based on a parametrix method which allows us to express $u_\ep$ as a solution of a fixed point problem. As a consequence, we obtain a strong estimate of $\nabla u_\ep$ in $L^p_{\rm loc}(\Om)$ for some $p>1$ close to $1$. However, this estimate cannot provide directly the desired limit $\nu$ of \refe{nu2}. since $p<2$. To overcome this difficulty we consider a truncation $\eta_\ep^k$ of $\nabla w_\ep$ which is bounded by $k>0$. Then, we can pass to the limit as $\ep$ tends to zero in the product $\eta_\ep^k\cdot\nabla u_\ep$ for a fixed $k$. Hence, thanks to the equi-integrability of $\nabla w_\ep$ we deduce the limit $\nu$ as $k$ tends to infinity.
\par
The proof is divided into four steps. In the first step we present the parametrix method which leads to a $L^p$-strong estimate of $\nabla u_\ep$. In the second step we determine the limit of the sequence $\eta_\ep^k\cdot\nabla u_\ep$ for a fixed $k>0$. In the third step we determine the limit $\nu$ and the limit equation \refe{eq2u}. together with \refe{muu}.. The fourth step is devoted to the proof of equality \refe{muu=}. and the corrector results \refe{corDue}. and \refe{corue}..
\par\medskip\noindent
{\it First step:} The parametrix method.
\par\ms\noindent
First, let us define a parametrix for the Laplace operator in $\Om$. To this end consider two sequences of functions $\ph_n, \psi_n$ in $C^\infty_c(\Om)$, such that 
\beq\label{phpsn}
\left\{\ba{l}
0\leq \ph_n,\psi_n\leq 1\quad\mbox{and}\quad\ph_n=1\mbox{ in }{\rm supp}\left(\psi_n\right),\quad\mbox{for any }n\geq 1,
\\ \ecart
\big\{n\geq 1\,:\,{\rm supp}\left(\psi_n\right)\cap K\neq{\rm\O}\big\}\mbox{ is finite},\quad\mbox{for any compact subset }K\subset \Om,
\\ \ecart
\dis \sum_{n\geq 1}\psi_n=1\quad\mbox{in }\Om.
\ea\right.
\eeq
Let $E$ be the fundamental solution of the Laplace operator in $\RR^N$. Then, the operator $P$ defined in $\D'(\Om)$ by
\beq\label{P}
P\zeta:=\sum_{n\geq 1}\psi_n\,E*\left(\ph_n\,\zeta\right),\quad\mbox{for }\zeta\in\D'(\Om),
\eeq
is a parametrix of the Laplace operator (see \cite{AlGe} Chapter~I, for further details) which satisfies
\beq\label{DP}
P(\De\zeta)=\zeta-K\zeta\quad\mbox{and}\quad\De(P\zeta)=\zeta-K'\zeta,\quad\mbox{for }\zeta\in\D'(\Om),
\eeq
where $K,K'$ are two $C^\infty$-kernel operators properly supported in $\Om$. Thanks to the Calder\`on-Zygmund regularity for the Laplace operator (see, e.g., \cite{Gey} Theorem~2.1, and the references therein) we also have for any $p>1$, and $s\in[0,2]$ such that $s+{1\over p}$ is not an integer,
\beq\label{regP}
\mbox{$P$ maps continuously $\D'(\Om)$ to $\D'(\Om)$, and $W^{-s,p}_{\rm loc}(\Om)$ to $W^{2-s,p}_{\rm loc}(\Om)$}.
\eeq
Then, applying \refe{DP}. to the solution $u_\ep$ of \refe{ue}. we have
\beq\label{ue11}
\ba{ll}
u_\ep=P(\De u_\ep)+Ku_\ep& =P\left[\div\big(u\left(\nabla w_\ep-\nabla w\right)\big)\right]+P\left[\div\big(\nabla w_\ep\,(u_\ep-u)\big)\right]
\\ \ecart
& +\,P\big[\div\left(u\nabla w\right)\big]+P\big(\xi_\ep\cdot\nabla u_\ep+b_\ep\cdot\nabla u_\ep-f\big)+Ku_\ep,
\ea
\eeq
Fix $p>1$ close enough to $1$ and $s\in ({N/p'},1)$. Since $u_\ep-u$ strongly converges to $0$ in $L^q(\Om)$ for any $q\in(2,2^*)$, the sequence $\div\big(\nabla w_\ep\,(u_\ep-u)\big)$ strongly converges to $0$ in $W^{-1,p}(\Om)$, hence by \refe{regP}. we have
\[
P\left[\div\big(\nabla w_\ep\,(u_\ep-u)\big)\right]\longrightarrow 0\quad\mbox{strongly in }W^{1,p}_{\rm loc}(\Om).
\]
Moreover, the sequence $\xi_\ep\cdot\nabla u_\ep+b_\ep\cdot\nabla u_\ep$ is bounded in $L^1(\Om)$, thus in $W^{-s,p}(\Om)$ since $s>{N/p'}$. Therefore, again by \refe{regP}. the sequence $\nabla P\big(\xi_\ep\cdot\nabla u_\ep+b_\ep\cdot\nabla u_\ep-f\big)$ is bounded in $W^{1-s,p}(\Om)^N$, and up to a subsequence strongly converges in $L^p_{\rm loc}(\Om)^N$. Hence, since
\[
\xi_\ep\cdot\nabla u_\ep+b_\ep\cdot\nabla u_\ep\confai\xi\cdot\nabla u+\nu+b\cdot\nabla u\quad\mbox{in }\D'(\Om),
\]
we deduce from \refe{ue11}. the strong estimate
\beq\label{ue12}
\ba{l}
\nabla u_\ep-\nabla P\left[\div\big(u\left(\nabla w_\ep-\nabla w\right)\big)\right]
\\ \ecart
=\nabla P\big[\div\left(u\nabla w\right)+\xi\cdot\nabla u+\nu+b\cdot\nabla u-f\big]+\nabla Ku+o_{L^p_{\rm loc}(\Om)^N}(1)
\\ \ecart
=\nabla P\big[\nu+b\cdot\nabla u+\div\left(b\,u\right)-f\big]+\nabla Ku+o_{L^p_{\rm loc}(\Om)^N}(1)
\quad\big(\mbox{since }\xi\cdot\nabla u=\div\left(u\,\xi\right)\big),
\ea
\eeq
where $o_{L^p_{\rm loc}(\Om)^N}(1)$ denotes a sequence which strongly converges to $0$ in $L^p_{\rm loc}(\Om)^N$.
On the other hand,  by \refe{DP}. and \refe{regP}. we have
\[
\ba{ll}
\nabla P\left[\div\big(u\left(\nabla w_\ep-\nabla w\right)\big)\right] & =\nabla P\left[\De\big(u\left(w_\ep-w\right)\big)\right]
-\nabla P\left[\div\big(\nabla u\,(w_\ep-w)\big)\right]
\\ \ecart
& =\nabla P\left[\De\big(u\left(w_\ep-w\right)\big)\right]+o_{L^p_{\rm loc}(\Om)^N}(1)
\\ \ecart
& =\nabla\big(u\left(w_\ep-w\right)\big)+o_{L^p_{\rm loc}(\Om)^N}(1)
\\ \ecart
& =u\left(\nabla w_\ep-\nabla w\right)+o_{L^p_{\rm loc}(\Om)^N}(1).
\ea
\]
Therefore, this combined with \refe{ue12}. yields
\beq\label{ue13}
\nabla u_\ep-u\left(\nabla w_\ep-\nabla w\right)=\nabla P\big[\nu+b\cdot\nabla u+\div\left(b\,u\right)-f\big]+\nabla Ku+o_{L^p_{\rm loc}(\Om)^N}(1).
\eeq
\par\medskip\noindent
{\it Second step:} Estimate of the sequence $\eta_\ep^k\cdot\nabla u_\ep$.
\par\ms\noindent
Set $\eta_\ep^k:=\nabla w_\ep\,1_{\{|\nabla w_\ep|<k\}}$, for a positive integer $k$. Let us determine the limit of $\eta_\ep^k\cdot\nabla u_\ep$ in $L^2_{\rm loc}(\Om)^N$. Using a diagonal extraction, there exists to a subsequence of $\ep$, still denoted by $\ep$, such that $\eta_\ep^k$ weakly converges to some $\eta^k$ in $L^\infty(\Om)^N$ for any $k$.
By the strong convergence~\refe{ue13}. combined with the weak convergence of $u\left(\nabla w_\ep-\nabla w\right)$ to $0$ in $L^p(\Om)^N$ (for $p$ close to $1$) we have
\[
\ba{l}
\eta_\ep^k\cdot\nabla u_\ep-\left(\eta_\ep^k-\eta^k\right)\cdot\left(\nabla w_\ep-\nabla w\right)u
\\ \ecart
\confai\eta^k\cdot\nabla P\big[\nu+b\cdot\nabla u+\div\left(b\,u\right)-f\big]+\eta^k\cdot\nabla Ku\quad\mbox{weakly in }L^p_{\rm loc}(\Om).
\ea
\]
Hence, we get that
\beq\label{sik=}
\si^k=\mu^k\,u+\eta^k\cdot\nabla P\big[\nu+b\cdot\nabla u+\div\left(b\,u\right)-f\big]+\eta^k\cdot\nabla Ku\quad\mbox{in }\Om,
\eeq
where
\beq\label{sikmuk}
\left\{\ba{ll}
\dis \si^k:=\lim_{\ep\to 0}\left[\eta_\ep^k\cdot\nabla u_\ep\right] & \mbox{weakly in }L^2(\Om),
\\ \ecart
\dis \mu^k:=\lim_{\ep\to 0}\left[\left(\eta_\ep^k-\eta^k\right)\cdot\left(\nabla w_\ep-\nabla w\right)\right] & \mbox{weakly in }L^2(\Om).
\ea\right.
\eeq
\par\medskip\noindent
{\it Third step:} Determination of $\nu$ and the limit equation \refe{eq2u}..
\par\ms\noindent
Starting from the limit equation \refe{e1}. we have by \refe{DP}.
\[
u=P\big[\nu+b\cdot\nabla u+\div\left(b\,u\right)-f\big]+Ku\quad\mbox{in }\Om,
\]
hence
\[
\eta^k\cdot\nabla u=\eta^k\cdot\nabla P\big[\nu+b\cdot\nabla u+\div\left(b\,u\right)-f\big]+\eta^k\cdot\nabla Ku\quad\mbox{in }\Om.
\]
Equating this with \refe{sik=}. we obtain
\beq\label{sik}
\si^k=\mu^k\,u+\eta^k\cdot\nabla u\quad\mbox{in }\Om.
\eeq
Now, let us pass to the limit as $k\to\infty$. By virtue of the equi-integrability of $\nabla w_\ep$ in $L^2(\Om)^N$ and by definition \refe{sikmuk}. the sequence $\mu^k$ strongly converges in $L^1(\Om)$ to the function $\mu$ of \refe{equiDwe}., $\eta^k$ strongly converges to $\nabla w$ in $L^2(\Om)^N$, and $\si^k$ strongly converges to $\nu+\nabla w\cdot\nabla u$ in $L^1(\Om)$. Then, up to a subsequence $\mu^k$ converges to $\mu$ a.e. in $\Om$, and by the Fatou lemma combined with equality \refe{sik}. we get
\beq\label{inemuu}
\int_\Om|\mu\,u|\,dx\leq\liminf_{k\to\infty}\int_\Om|\mu^k\,u|\,dx\leq\liminf_{k\to\infty}\int_\Om|\si^k-\eta^k\cdot\nabla u|\,dx
=\int_\Om|\nu|\,dx.
\eeq
We deduce from \refe{inemuu}. and \refe{sik}. that $\mu\,u\in L^1(\Om)$ and
\beq\label{mu=}
\nu=\mu\,u\quad\mbox{in }\Om,
\eeq
which yields the limit equation \refe{eq2u}..
\par
It remains to prove the inequality of \refe{muu}.. Let $v\in L^\infty(\Om)$ and $t\in\RR$. By \refe{equiDwe}., \refe{nu2}. and \refe{mu=}. we have
\beq\label{corDue1}
\ba{l}
\dis \int_\Om\big|\nabla u_\ep-\nabla u-\left(\nabla w_\ep-\nabla w\right)t\,v\big)\big|^2\,dx
\\ \ecart
\dis =\int_\Om|\nabla u_\ep-\nabla u|^2\,dx+t^2\int_\Om|\nabla w_\ep-\nabla w|^2\,v^2\,dx-2\,t\int_\Om\nabla u_\ep\cdot\left(\nabla w_\ep-\nabla w\right)v\,dx+o(1)
\\ \ecart
\dis =\langle f,u\rangle_{H^{-1}(\Om),H^1_0(\Om)}-\int_\Om|\nabla u|^2\,dx+t^2\int_\Om\mu\,v^2\,dx-2\,t\int_\Om\mu\,u\,v\,dx+o(1), 
\ea
\eeq
hence
\[
t^2\int_\Om\mu\,v^2\,dx-2\,t\int_\Om\mu\,u\,v\,dx+\langle f,u\rangle_{H^{-1}(\Om),H^1_0(\Om)}-\int_\Om|\nabla u|^2\,dx\geq 0,\quad\forall\,t\in\RR.
\]
This implies that
\beq\label{inemuuv}
\left(\int_\Om\mu\,u\,v\,dx\right)^2\leq\left(\langle f,u\rangle_{H^{-1}(\Om),H^1_0(\Om)}-\int_\Om|\nabla u|^2\,dx\right)\int_\Om\mu\,v^2\,dx.
\eeq
Let $T_k$, $k>0$, be a function in $C^1(\RR)$ such that
\beq\label{Tk}
0\leq T'_k\leq 1\quad\mbox{and}\quad\left\{\ba{ll} T_k(t)=t & \mbox{if }|t|\leq k \\ |T_k(t)|=k+1 & \mbox{if }|t|\geq k+2.\ea\right.
\eeq
Putting $v=T_k(u)$ as test function in \refe{inemuuv}. and using that $T_k(u)^2\leq u\,T_k(u)$, we get
\[
\left(\int_\Om\mu\,u\,T_k(u)\,dx\right)^2\leq\left(\langle f,u\rangle_{H^{-1}(\Om),H^1_0(\Om)}-\int_\Om|\nabla u|^2\,dx\right)\int_\Om\mu\,u\,T_k(u)\,dx,
\]
hence
\beq\label{ineTku}
\int_\Om\mu\,u\,T_k(u)\,dx\leq\langle f,u\rangle_{H^{-1}(\Om),H^1_0(\Om)}-\int_\Om|\nabla u|^2\,dx.
\eeq
Since $u\,T_k(u)$ is a nondecreasing nonnegative sequence which converges to $u^2$ a.e. in $\Om$, the Beppo-Levi theorem applied to \refe{ineTku}. thus gives inequality \refe{muu}..
\par\medskip\noindent
{\it Fourth step:} Proof of equality \refe{muu=}. and of the corrector results \refe{corDue}., \refe{corue}..
\par\ms\noindent
Assume that $b\in L^q(\Om)^N$, where $q>2$ if $N=2$ and $q=N$ if $N>2$. Let $\ph_n$ be a sequence in $C^1_0(\RR)$ which strongly converges to $u$ in $H^1_0(\Om)$ and a.e. in $\Om$, and such that $|\nabla\ph_n|$ is dominated by a fixed function in $L^2(\Om)$. Putting the truncation function $T_k(\ph_n)$ \refe{Tk}. in the limit equation \refe{eq2u}. we have
\[
\ba{l}
\dis \int_\Om\nabla u\cdot\nabla T_k(\ph_n)\,dx+\int_\Om b\cdot\nabla u\,T_k(\ph_n)\,dx-\int_\Om b\cdot\nabla T_k(\ph_n)\,u\,dx+\int_\Om\mu\,u\,T_k(\ph_n)\,dx
\\ \ecart
\dis =\big\langle f,T_k(\ph_n)\big\rangle_{H^{-1}(\Om),H^1_0(\Om)}.
\ea
\]
Since $b\cdot\nabla u$, $\mu\,u\in L^1(\Om)$ and $b\,u\in L^2(\Om)^N$ (as a consequence of $b\in L^q(\Om)^N$), we can pass to the limit as $n\to\infty$ in the previous equality owing to the Lebesgue dominated convergence theorem, which yields
\beq\label{tfTku}
\ba{l}
\dis \int_\Om\nabla u\cdot\nabla T_k(u)\,dx+\int_\Om b\cdot\nabla u\,T_k(u)\,dx-\int_\Om b\cdot\nabla T_k(u)\,u\,dx+\int_\Om\mu\,u\,T_k(u)\,dx
\\ \ecart
\dis =\big\langle f,T_k(u)\big\rangle_{H^{-1}(\Om),H^1_0(\Om)}.
\ea
\eeq
Then, using that $|T_k(u)|\leq|u|$, $0\leq T'_k(u)\leq 1$, $T_k(u)$ strongly converges to $u$ in $H^1_0(\Om)$, and that $b\,u\in L^2(\Om)^N$, $\mu\,u^2\in L^1(\Om)$, and passing to the limit as $k\to\infty$ owing to the Lebesgue dominated convergence theorem we get
\[
\int_\Om|\nabla u|^2\,dx+\int_\Om b\cdot\nabla u\,u\,dx-\int_\Om b\cdot\nabla u\,u\,dx+\int_\Om\mu\,u^2\,dx
=\langle f,u\rangle_{H^{-1}(\Om),H^1_0(\Om)}\,,
\]
which is \refe{muu=}.. Moreover, the proof of equality \refe{muu=}. with $f=0$ shows that there exists a unique solution $u\in H^1_0(\Om)$ of equation \refe{eq2u}., with $\mu\,u^2\in L^1(\Om)$.
\par
It remains to prove the corrector results.
By the estimate \refe{corDue1}. with $v=T_k(u)$ and $t=1$, combined with equality \refe{muu=}. we have
\beq\label{corDue2}
\ba{l}
\dis \lim_{k\to\infty}\;\lim_{\ep\to 0}\left(\int_\Om\big|\nabla u_\ep-\nabla u-\left(\nabla w_\ep-\nabla w\right)T_k(u)\big|^2\,dx\right)
\\ \ecart
\dis =\lim_{k\to\infty}\left(\int_\Om\mu\big(u-T_k(u)\big)^2\,dx\right)=0.
\ea
\eeq
On the other hand, let $p\in[1,2)$ if $N=2$ and $p=N'$ if $N>2$, and consider an open set $\om\Subset\Om$. By the H\"older inequality we have
\beq\label{corDue3}
\ba{l}
\dis \int_\om\big|\nabla u_\ep-\nabla u-\left(\nabla w_\ep-\nabla w\right)u\big|^p\,dx
\\ \ecart
\dis \leq 2^{p-1}\left(\int_\Om\big|\nabla u_\ep-\nabla u-\left(\nabla w_\ep-\nabla w\right)T_k(u)\big|^p
+\int_\om\left|\nabla w_\ep-\nabla w\right|^p\,\big|u-T_k(u)\big|^p\,dx\right)
\\ \ecart
\dis \leq c\left(\int_\Om\big|\nabla u_\ep-\nabla u-\left(\nabla w_\ep-\nabla w\right)T_k(u)\big|^2\right)^{p\over 2}
+c\left(\int_\om\big|u-T_k(u)\big|^{2p\over 2-p}\,dx\right)^{1-{p\over 2}}
\\ \ecart
\dis \leq c\left(\int_\Om\big|\nabla u_\ep-\nabla u-\left(\nabla w_\ep-\nabla w\right)T_k(u)\big|^2\right)^{p\over 2}
+c\left(\int_{\{|u|>k\}\cap\om}|u|^{2p\over 2-p}\,dx\right)^{1-{p\over 2}}.
\ea
\eeq
Since $u\in L^{2p\over 2-p}(\om)$ by the Sobolev embedding, passing successively to the limits $\ep\to 0$ and $k\to\infty$ in \refe{corDue3}. owing to convergence \refe{corDue2}. we obtain the strong convergence \refe{corDue}..
\par
Let $r\in[1,p)$. Since $w_\ep-w$ strongly converges to $0$ in $L^{2r\over 2-r}(\om)$, by the H\"older inequality the sequence $\left(w_\ep-w\right)\nabla u$ strongly converges to $0$ in $L^r(\om)^N$. Finally, this combined with \refe{corDue}. implies the corrector result \refe{corue}..
\subsection{A counter-example}\label{s.cex}
In this section $\Om$ is a regular bounded open set of $\RR^2$, and $Y:=(-{1\over 2},{1\over 2})^2$.
For fixed $R\in(0,{1\over 2})$ and $\mu>0$, let $r_\ep\in(0,R)$ be defined by the equality
\beq\label{remu}
{2\pi\over\ep^2\left|\ln r_\ep\right|}=\mu.
\eeq
Let $W_\ep$ be the $Y$-periodic function and $w_\ep$ be the $\ep Y$-periodic function defined by
\beq\label{Wewe}
W_\ep(y):=\left\{\ba{cl}
\dis\frac{\ln r-\ln r_\ep}{\ln R-\ln r_\ep} & \mbox{if }r:=|y|\in\,(r_\ep,R)
\\ \ecart
0 & \mbox{si }r\leq r_\ep
\\ \ecart
1 & \mbox{si }r\geq R,
\ea\right.
y\in Y,
\quad w_\ep(x):=W_\ep\left({x\over\ep}\right),\quad x\in\RR^2.
\eeq
Note that by \refe{remu}. we have
\beq\label{Wemu}
{1\over\ep^2}\int_Y|\nabla W_\ep|^2\,dy={2\pi\over\ep^2\ln\left(R/r_\ep\right)}\limo\mu.
\eeq
We then consider the drift $b_\ep$ defined by
\beq\label{bewe}
b_\ep(x):=\nabla w_\ep(x)={1\over\ep}\,\nabla W_\ep\left({x\over\ep}\right),\quad\mbox{for }x\in\RR^2.
\eeq
Taking into account \refe{remu}. it is easy to check that
\beq\label{cwe}
w_\ep\confai 1\quad\mbox{weakly in }H^1(\Om)\mbox{ and weakly-$*$ in }L^\infty(\Om).
\eeq
\par
Let $f$ be a non-zero function in $L^2(\Om)$. We study the asymptotic behavior of the equation \refe{ue}. with the drift $b_\ep$ of \refe{bewe}., i.e. 
\beq\label{uewe}
-\,\De u_\ep+\nabla w_\ep\cdot\nabla u_\ep+{\rm div}\left(\nabla w_\ep\,u_\ep\right)=f\quad\mbox{in }\D'(\Om).
\eeq
We have the following result:
\begin{Thm}\label{thm.cex}
The solution $u_\ep$ of \refe{uewe}. weakly converges in $H^1_0(\Om)$ to the solution $u$ of the equation
\beq\label{uga}
-\De u+\ga\,u=f\quad\mbox{in }\D'(\Om),\quad\mbox{where}\quad\ga:={3\left(e^2-1\right)\over 4\left(e^2+1\right)}\,\mu<\mu.
\eeq
\end{Thm}
\begin{Rem}\label{rem.cex}
Using the periodicity we can check that the sequence $|b_\ep|^2=|\nabla w_\ep|^2$ converges in the weak-$*$ sense of measures on $\Om$ -- but not weakly in $L^1(\Om)$ -- to the constant $\mu$ defined by \refe{remu}.. Theorem~\ref{thm.cex} can thus be regarded as a counter-example to the statement of Theorem~\ref{thm1} without the equi-integrability assumption on the drift $b_\ep$ in $L^2(\Om)^2$.  Indeed, the conclusion of Theorem~\ref{thm1} would give a limit equation \refe{uga}., with $\ga=\mu$.
\end{Rem}
\noindent
{\bf Proof of Theorem~\ref{thm.cex}.} The proof is divided into two steps. In the first step we construct an oscillating test function $z_\ep$ which solves equation \refe{zewe}. below. In the second step we determine the limit equation \refe{uga}..
\par\bs\noindent
{\it First step:} Construction of an oscillating test function.
\par\ms\noindent
Denote by $Q_r$ the disk of radius $r$ centered at the origin.  Consider the unique solution $Z_\ep$ in $H^1(Q_R)$ of the equation
\beq\label{Ze}
\left\{\ba{rll}
\dis -\,{1\over\ep^2}\,\De Z_\ep+{1\over\ep^2}\,|\nabla W_\ep|^2\,Z_\ep & \dis ={1_{Q_R}\over |Q_R|} & \mbox{in }Q_R
\\ \ecart
\dis {\partial Z_\ep\over\partial n} & =0 & \mbox{on }\partial Q_R.
\ea\right.
\eeq
The function $Z_\ep$ is radial and can be computed explicitly. Using the Laplace operator in polar coordinates and
$|\nabla W_\ep|^2=\al_\ep^2\,r^{-2}\,1_{Q_R\setminus\bar{Q}_{r_\ep}}$, we get
\beq\label{Zer}
Z_\ep(r)=\left\{\ba{rl}
\dis -{\ep^2\over 4\pi R^2}\,r^2+c_\ep & \mbox{if }r\in(0,r_\ep]
\\ \ecart
\dis a_\ep\,r^{\al_\ep}+b_\ep\,r^{-\al_\ep}+{\ep^2\over\pi R^2\left(\al_\ep^2-3\right)}\,r^2 & \mbox{if }r\in(r_\ep,R],
\ea\right.
\quad\mbox{where }\al_\ep:={1\over\ln(R/r_\ep)}.
\eeq
The constants $a_\ep$, $b_\ep$, $c_\ep$ are determined owing to the boundary condition on $\partial Q_R$ and to the transmission conditions on $\partial Q_{r_\ep}$, i.e.
\beq\label{tbcZe}
Z_\ep'(R)=0\quad\mbox{and}\quad Z_\ep(r_\ep^+)=Z_\ep(r_\ep^-),\quad Z_\ep'(r_\ep^+)=Z_\ep'(r_\ep^-).
\eeq
We extend $Z_\ep$ by the constant value $Z_\ep(R)$ in $Y\setminus\bar{Q}_R$, and by $Y$-periodicity in the whole space~$\RR^2$. The $Y$-periodic extension is still denoted by $Z_\ep$.
An explicit computation combined with \refe{remu}. yields
\beq\label{bZ}
Z_\ep\longrightarrow\bar{Z}:={4\left(e^2+1\right)\over 3\left(e^2-1\right)}\,{1\over\mu}\quad\mbox{strongly in }H^1_\sharp(Y).
\eeq
As a consequence of \refe{Ze}., \refe{Zer}. the rescaled function $z_\ep(x):=Z_\ep({x\over\ep})$ is solution of the equation
\beq\label{zewe}
-\,\De z_\ep+|\nabla w_\ep|^2\,z_\ep=\chi^\sharp_{Q_R}\left({x\over\ep}\right)\quad\mbox{in }\D'(\RR^2),
\eeq
where $\chi^\sharp_{Q_R}$ is the $Y$-periodic function agreeing with ${1_{Q_R}\over|Q_R|}$ in the period cell $Y$.
Moreover, the following convergences hold
\beq\label{czechep}
z_\ep\confai\bar{Z}\;\;\mbox{weakly in }H^1(\Om)\quad\mbox{and}\quad\chi^\sharp_{Q_R}\left({x\over\ep}\right)\confai 1\;\;\mbox{weakly-$*$ in }L^\infty(\Om),
\eeq
where the constant $\bar{Z}$ is defined by \refe{bZ}..
\par\bs\noindent
{\it Second step:} Determination of the limit equation \refe{uga}..
\par\ms\noindent
Define the function $v_\ep:=e^{1-w_\ep}\,u_\ep$. Then, equation \refe{uewe}. is equivalent to
\beq\label{veve}
-\,\De v_\ep+|\nabla w_\ep|^2\,v_\ep=e^{1-w_\ep}\,f\quad\mbox{in }\D'(\Om).
\eeq
G.~Dal Maso, A.~Garroni \cite{DMGa} proved that this class of equations is stable under homogenization. In the present case, the use of the oscillating test function $z_\ep$ will allow us to obtain the limit equation \refe{uga}..
\par
On the one hand, choosing $v=w_\ep$ in \refe{vue}. we get
\beq
\int_\Om|\nabla w_\ep|^2\,u^2_\ep\,dx-\int_\Om\nabla w_\ep\cdot\nabla u_\ep\,u_\ep\,dx
=\int_\Om|\nabla u_\ep|^2\,w_\ep\,dx-\int_\Om f\,w_\ep\,u_\ep\,dx\leq c,
\eeq
since $u_\ep$ is bounded in $H^1_0(\Om)$ and $0\leq w_\ep\leq 1$. Then, by the Cauchy-Schwarz inequality we have
\beq\label{eDweue}
\ba{ll}
\dis \int_\Om|\nabla w_\ep|^2\,u^2_\ep\,dx
& \dis \leq c+c\left(\int_\Om|\nabla u_\ep|^2\,dx\right)^{1\over 2}\left(\int_\Om|\nabla w_\ep|^2\,u^2_\ep\,dx\right)^{1\over 2}
\\ \ecart
& \dis \leq c+c'\left(\int_\Om|\nabla w_\ep|^2\,u^2_\ep\,dx\right)^{1\over 2},
\ea
\eeq
hence $u_\ep\nabla w_\ep$ is bounded in $L^2(\Om)^2$. This combined with convergence \refe{cwe}. implies that $v_\ep$ weakly converges to $u$ in $H^1_0(\Om)$.
\par
On the other hand, for $\ph\in C^\infty_c(\Om)$, putting the functions $\ph\,z_\ep$ in \refe{veve}. and $\ph\,v_\ep$ in \refe{zewe}., taking the difference of the two equalities, and passing to the limit owing to convergences \refe{czechep}. we obtain the equality
\beq\label{bZu}
\int_\Om\nabla u\cdot\nabla\ph\,\bar{Z}\,dx+\int_\Om\ph\,u\,dx=\int_\Om f\,\ph\,\bar{Z}\,dx,\quad\mbox{for any }\ph\in C^\infty_c(\Om).
\eeq
which is the variational formulation of equation \refe{uga}., with $\ga=\bar{Z}^{-1}$. \cqfd
\section{A Stokes equation with a drift term}\label{s.Sto}
\subsection{The classical case}\label{ss.TS}
In \cite{Tar1,Tar2} L.~Tartar noted that the nonlinear term of the three-dimensional Navier-Stokes equation for the divergence free velocity $u$ reads as
\beq\label{nlt3NS}
\left(u\cdot\nabla\right)u=\Div\left(u\otimes u\right)=\curl\left(u\right)\times u+\nabla\left({\tex{1\over 2}}\,|u|^2\right).
\eeq
This led him to study the perturbed Stokes equation
\beq\label{3pS}
-\,\De u+\curl\left(v\right)\times u+\nabla p=f,
\eeq
where a given vector-valued function $v$ replaced the velocity $u$ of the Navier-Stokes equation.
The equivalent of transformation \refe{nlt3NS}. in two-dimension is
\beq\label{nlt2NS}
\ba{ll}
& \dis \Div\left(u\otimes u\right)=\curl\left(u\right)Ju+\nabla\left({\tex{1\over 2}}\,|u|^2\right),
\\ \ecart
\mbox{where} & \curl\left(u\right):=\partial_1 u_2-\partial_2 u_1\quad\mbox{and}\quad J:=\begin{pmatrix} 0 & -1 \\ 1 & 0\end{pmatrix}.
\ea
\eeq
More generally equality \refe{nlt2NS}. extends for any divergence free functions $u,v$ to the following one
\beq\label{Duv}
\curl\left(v\right)Ju=\Div\left(v\otimes u\right)+\left(Du\right)^T v-\nabla\left(v\cdot u\right).
\eeq
Similarly to \refe{3pS}. this leads us to the two-dimensional perturbed Stokes equation
\beq\label{2pS}
-\,\De u+\curl\left(v\right)Ju+\nabla p=f.
\eeq
\par\bs
Let $\Om$ be a bounded domain of $\RR^2$.  Let $v_\ep$ be a sequence in $L^\infty(\Om)^2$ and let $f$ be a distribution in $H^{-1}(\Om)^2$. Consider the perturbed Stokes equation
\beq\label{pSe}
\left\{\ba{rll}
-\,\De u_\ep+\curl\left(v_\ep\right)Ju_\ep+\nabla p_\ep & =f & \mbox{in }\Om
\\ \ecart
\div\left(u_\ep\right) & =0 &  \mbox{in }\Om
\\ \ecart
u_\ep & =0 &  \mbox{on }\partial\Om.
\ea\right.
\eeq
\par
In the three-dimensional case where $\curl\left(v_\ep\right)\times u_\ep$ replaces $\curl\left(v_\ep\right)Ju_\ep$, L.~Tartar \cite{Tar2} derived a Stokes equation with a Brinkman law under the assumption that $v_\ep$ is bounded in $L^3(\Om)^3$ (see Introduction). Mimicking the Tartar approach in dimension two we can derive a similar homogenized equation using the test function $w_\ep^\la$, for $\la\in\RR^2$, solution of the Stokes equation
\beq\label{Sewel}
\left\{\ba{rll}
-\,\De w_\ep^\la+\Div\big((v_\ep-v)\otimes\la\big)+\nabla q_\ep^\la & =0 & \mbox{in }\Om
\\ \ecart
\div\left(w_\ep^\la\right) & =0 &  \mbox{in }\Om
\\ \ecart
w_\ep^\la & =0 & \mbox{on }\partial\Om.
\ea\right.
\eeq
Then, we have the following result:
\begin{Thm}\label{thmT2}
Assume that $v_\ep$ is bounded in $L^r(\Om)^2$, with $r>2$. Then, the solution $u_\ep$ of \refe{pSe}. weakly converges in $H^1_0(\Om)$ to the solution $u$ of the Brinkman equation
\beq\label{eqBT}
\left\{\ba{rll}
-\,\De u+\curl\left(v\right)Ju+\nabla p+{M}u & =f & \mbox{in }\Om
\\ \ecart
\div\left(u\right) & =0 &  \mbox{in }\Om
\\ \ecart
u & =0 &  \mbox{on }\partial\Om,
\ea\right.
\eeq
where ${M}$ is the positive definite symmetric matrix-valued function defined by
\beq\label{MTr}
\left\{\ba{rlll}
\dis (Dw_\ep^\la)^T v_\ep & \confai {M}\la & \mbox{weakly in }L^{2r\over 2+r}(\Om)^2\mbox{ and in }L^{r\over 2}_{\rm loc}(\Om)^2
\\ \ecart
\dis Dw_\ep^\la\cdot Dw_\ep^\mu & \confai {M}\la\cdot\mu & \mbox{weakly-$*$ in }\M(\Om)^2\mbox{ and in }L^{r\over 2}_{\rm loc}(\Om)^2,
\ea\right.
\quad\mbox{for }\la,\mu\in\RR^2.
\eeq
Moreover, the zero-order term of \refe{eqBT}. is given by the convergences
\beq\label{Dueve}
\left\{\ba{rlll}
(Du_\ep)^T(v_\ep-v) & \confai {M}u & \mbox{weakly in }L^{2r\over 2+r}(\Om)^2
\\ \ecart
Du_\ep:Dw_\ep^\la & \confai {M}u\cdot\la & \mbox{weakly-$*$ in }\M(\Om)\mbox{ and in }L^{2r\over 2+r}_{\rm loc}(\Om)^2.
\ea\right.
\eeq
\end{Thm}
\begin{proof} By the representation formula \refe{Duv}. we have
\beq\label{Duve}
\curl\left(v_\ep\right)Ju_\ep=\left(Du_\ep\right)^T v_\ep+\Div\left(v_\ep\otimes u_\ep\right)-\nabla\left(v_\ep\cdot u_\ep\right).
\eeq
Hence, the variational formulation of \refe{pSe}. reads as
\beq\label{vfpSe}
\ba{r}
\dis \int_\Om Du_\ep:D\ph\,dx+\int_\Om(Du_\ep)^T v_\ep\cdot\ph\,dx-\int_\Om(v_\ep\otimes u_\ep):D\ph\,dx=\langle f,u_\ep\rangle_{H^{-1}(\Om)^2,H^1_0(\Om)^2}\,,
\\ \ecart
\mbox{for any }\ph\in H^1_0(\Om)^2,\ \div\left(\ph\right)=0.
\ea
\eeq
By the Lax-Milgram theorem there exists a unique divergence free function $u_\ep\in H^1_0(\Om)^2$ solution of \refe{vfpSe}..
Then, putting the velocity $u_\ep$ as test function in \refe{vfpSe}. it follows that
\beq\label{estDue}
\int_\Om|Du_\ep|^2\,dx=\langle f,u_\ep\rangle_{H^{-1}(\Om)^2,H^1_0(\Om)^2}\,,
\eeq
which implies that $u_\ep$ is bounded in $H^1_0(\Om)^2$. Let $\om$ be a regular domain of $\Om$. Applying \refe{vfpSe}. to divergence free functions in $H^1_0(\om)^2$, there exists a unique $p_\ep$ in $L^2(\om)/\RR$ such that equation \refe{pSe}. holds in $\D'(\om)^2$. Moreover, by \refe{Duve}. and the boundedness of $v_\ep$ in $L^r(\Om)^2$ the sequence $\nabla p_\ep$ is bounded in $H^{-1}(\om)^2$. Hence, due to the regularity of $\om$ the sequence $p_\ep$ is bounded in $L^2(\om)$.  Then, considering an exhaustive sequence of regular domains the union of which is $\Om$, we can construct in $\Om$ a pressure $p_\ep$ which is bounded in $L^2_{\rm loc}(\Om)$. Therefore, up to a subsequence the following convergences hold
\beq\label{conuepe}
\left\{\ba{rll}
u_\ep & \confai u & \mbox{weakly in }H^1_0(\Om)^2
\\ \ecart
p_\ep & \confai p & \mbox{weakly in }L^2_{\rm loc}(\Om)/\RR,
\ea\right.
\eeq
\par
Now, in view of \refe{Duve}. it is enough to determine the limit of the term $\left(Du_\ep\right)^T v_\ep$.
By the regularity results for the Stokes equation (see, e.g., \cite{Lad} Theorem~2, p.~67) the sequences $w_\ep^\la$ and $q_\ep^\la$ satisfy
\beq\label{cwqel}
\left\{\ba{rlll}
w_\ep^\la & \confai 0 & \mbox{weakly in }H^1(\Om)^2\mbox{ and in }W^{1,r}_{\rm loc}(\Om)^2
\\ \ecart
q_\ep^\la & \confai 0 & \mbox{weakly in }L^2(\Om)/\RR\mbox{ and in }L^r_{\rm loc}(\Om)/\RR.
\ea\right.
\eeq
which imply convergence \refe{MTr}.. Let $\ph\in C^\infty_c(\Om)$. Following the Tartar method we put $\ph\,w_\ep^\la$ in equation \refe{pSe}. and $\ph\,u_\ep$ in equation \refe{Sewel}.. Then, from the representation \refe{Duve}., the convergences \refe{conuepe}., \refe{cwqel}. and the boundedness of $v_\ep$ in $L^r(\Om)$ we deduce that
\[
\left\{\ba{l}
\dis \int_\Om Du_\ep:D w_\ep^\la\,\ph\,dx-\int_\Om(v_\ep\otimes u_\ep):Dw_\ep^\la\,dx=o(1)
\\ \ecart
\dis \int_\Om Dw_\ep^\la:Du_\ep \,\ph\,dx-\int_\Om\big((v_\ep-v)\otimes\la\big):Du_\ep\,\ph\,dx=o(1),
\ea\right.
\quad\mbox{for any }\ph\in C^\infty_c(\Om),
\]
hence
\beq\label{uewel}
\left\{\ba{rl}
Du_\ep:D w_\ep^\la-(Dw_\ep^\la)^T v_\ep\cdot u_\ep & \confai 0
\\ \ecart
(Du_\ep)^T(v_\ep-v)\cdot\la-(Dw_\ep^\la)^T v_\ep\cdot u_\ep & \confai 0
\ea\right.
\quad\mbox{in }\D'(\Om).
\eeq
By virtue of the strong convergence of $u_\ep$ in any $L^s(\Om)^2$ space for $s\in(1,\infty)$, convergences \refe{uewel}. and \refe{MTr}. imply \refe{Dueve}..
This combined with \refe{Duve}. yields finally the limit problem \refe{eqBT}..
\end{proof}
\begin{Rem}
It can be shown that
\beq\label{muHS}
M(x)=\int_{S^1}\big[{\rm tr}\left(\mbox{\bmu}\left(x,d\xi\right)\right)-\mbox{\bmu}\left(x,d\xi\right)\xi\cdot\xi\big]\,\xi\otimes\xi
\eeq
where {\bmu} is the matrix-valued $H$-measure of the sequence $v_\ep$ (see \cite{Tar3,Tar4}).
\end{Rem}
\par
The case where $v_\ep$ is only bounded in $L^2(\Om)^2$ is much more delicate. On the one hand, under additional assumptions we will extend the Tartar result when $v_\ep$ is bounded and equi-integrable in $L^2(\Om)^2$. On the other hand, we will give an example of a sequence $v_\ep$ for which the homogenized Brinkman equation is not the one obtained by the Tartar procedure.
\subsection{The case under an equi-integrability condition}\label{ss.equiS}
In this section we make the following weaker assumption on the drift,
\beq\label{equive}
v_\ep\confai v\;\;\mbox{weakly in }L^2(\Om)^2\quad\mbox{and}\quad v_\ep\mbox{ is equi-integrable in }L^2(\Om)^2.
\eeq
Then, we have the following extension of Theorem~\ref{thmT2}:
\begin{Thm}\label{thm2}
\hfill
\par\ss\noindent
$i)$ Under the equi-integrability assumption \refe{equive}. the solution $u_\ep$ of \refe{ue}. weakly converges in $H^1_0(\Om)$ to the solution $u$ of equation \refe{eqBT}. with 
\beq\label{MTuu}
\int_\Om {M}u\cdot u\,dx\leq\langle f,u\rangle_{H^{-1}(\Om)^2,H^1_0(\Om)^2}-\int_\Om|Du|^2\,dx,
\eeq
where ${M}$ is the positive definite symmetric matrix-valued function defined by
\beq\label{MT}
\left\{\ba{rlll}
\dis (Dw_\ep^\la)^T v_\ep & \confai {M}\la & \mbox{weakly in }L^1(\Om)^2
\\ \ecart
\dis Dw_\ep^\la:Dw_\ep^\mu & \confai {M}\la\cdot\mu & \mbox{weakly-$*$ in }\M(\Om)^2,
\ea\right.
\quad\mbox{for }\la,\mu\in\RR^2.
\eeq
\par\noindent
$ii)$ Also assume that $\Om$ has a Lipschitz boundary, $v\in L^r(\Om)^2$, with $r>2$, and ${M}\in L^m(\Om)^{2\times 2}$, with $m>1$. Then, we have the equality
\beq\label{MTuu=}
\int_\Om|Du|^2\,dx+\int_\Om {M}u\cdot u\,dx=\langle f,u\rangle_{H^{-1}(\Om)^2,H^1_0(\Om)^2}\,,
\eeq
and there exists a unique solution $u\in H^1_0(\Om)^2$ of equation \refe{eqBT}., with ${M}u\cdot u\in L^1(\Om)$.
\\
Moreover, we have the corrector result
\beq\label{corueWe}
u_\ep-u-W_\ep\,u\;\longrightarrow\;0\quad\mbox{strongly in }W^{1,1}(\Om)^2,
\eeq
where $W_\ep$ is the matrix-valued function defined by
\beq\label{We}
W_\ep\,\la:=w_\ep^\la,\quad\mbox{for }\la\in\RR^2.
\eeq
\end{Thm}
\begin{Rem}
Contrary to Theorem~\ref{thm1}, in the part $ii)$ of Theorem~\ref{thm2} we need to assume a higher integrability for the matrix-valued ${M}$. Indeed, we cannot apply a truncation principle on ${M}u\cdot u$. Moreover, the regularity of $\Om$ is necessary to obtain the density of the smooth divergence free functions in the space of the divergence free functions of $H^1_0(\Om)^2$.
\end{Rem}
\noindent
{\bf Proof of Theorem~\ref{thm2}.}
As in the proof of Theorem~\ref{thm2} the sequence $u_\ep$ is bounded in $H^1_0(\Om)^2$, and thus in any $L^s(\Om)^2$ space. Then, in view of \refe{Duve}. and \refe{pSe}. together with the boundedness of $u_\ep$ and $v_\ep$ the sequence $\nabla p_\ep$ is bounded in $L^1(\Om)^2+W^{-1,r}(\Om)^2$ for any $r\in(1,2)$. Hence, thanks to the embedding of $L^1_{\rm loc}(\Om)$ into $W^{-\si,r}_{\rm loc}(\Om)$ for any $r>1$ and $\si>2/r'$, 
the sequence $p_\ep$ is bounded in $L^r_{\rm loc}(\Om)/\RR$ for any $r\in(1,2)$. Therefore, up to a subsequence we have the convergences
\beq\label{conueper}
\left\{\ba{rlll}
u_\ep & \confai u & \mbox{weakly in }H^1_0(\Om)^2 &
\\ \ecart
p_\ep & \confai p & \mbox{weakly in }L^r_{\rm loc}(\Om)/\RR, & \mbox{for any }r\in(1,2).
\ea\right.
\eeq
The problem is to determine the vector-valued distribution $\nu$ defined by
\beq\label{nuS1}
\curl\left(v_\ep\right)Ju_\ep-\curl\left(v\right)Ju\confai \nu\quad\mbox{in }\D'(\Om)^2.
\eeq
Taking into account the representation formula \refe{Duve}. and the equi-integrability of $v_\ep$ in $L^2(\Om)^2$, $\nu$ is actually in $L^1(\Om)^2$, and is given by
\beq\label{nuS2}
(Du_\ep)^T(v_\ep-v)\confai \nu\quad\mbox{weakly in }L^1(\Om)^2,
\eeq
so that $u$ is solution of the equation
\beq\label{eqSnu}
-\,\De u+\nu+\curl\left(v\right)Ju+\nabla p=f\quad\mbox{in }\D'(\Om).
\eeq
\par
From now on the proof follows the same scheme as the one of Theorem~\ref{thm1} using a representation of the velocity and the pressure owing to the parametrix $P$ of \refe{P}..
The proof is divided into fifth steps. The first step deals with a double parametrix method for both $u_\ep$ and $p_\ep$, which allows us to derive a strong approximation of $Du_\ep$. In the second step we compute the limit $\si^k$ of the sequence $(Du_\ep)^T v_\ep^k$, where $v_\ep^k$ is a truncation of $v_\ep$ for a fixed $k>0$. In the third step we obtain the limit equation \refe{eqBT}.. In the fourth step we prove inequality \refe{MTuu}.. The fifth step is devoted to the proof of equality \refe{MTuu}. and the corrector result \refe{corueWe}..
\par\bs\noindent
{\it First step:} The double parametrix method.
\par\ms\noindent
Consider the parametrix $P$ \refe{P}. for the Laplace operator. Abusively we denote by $\De$ the vector-valued Laplace operator as well as by $P$ the associated vector-valued parametrix each component of which is defined by \refe{P}..
Taking the divergence of equation \refe{pSe}. we have
\[
\De p_\ep=\div\left(f\right)-\div\left(\curl\left(v_\ep\right)Ju_\ep\right)\quad\mbox{in }\Om,
\]
hence by \refe{DP}.
\beq\label{Ppe}
p_\ep=P\big[\div\left(f\right)-\div\left(\curl\left(v_\ep\right)Ju_\ep\right)\big]+Kp_\ep\quad\mbox{in }\Om.
\eeq
Substituting $p_\ep$ by the right-hand side of \refe{Ppe}. in \refe{pSe}. it follows that
\[
\De u_\ep=\curl\left(v_\ep\right)Ju_\ep-\nabla P\big(\div\left(\curl\left(v_\ep\right)Ju_\ep\right)\big)+\nabla P\big(\div\left(f\right)\big)-f+\nabla Kp_\ep
\quad\mbox{in }\Om,
\]
hence again by \refe{DP}. we have in $\Om$
\beq\label{ue21}
u_\ep=P\left[\curl\left(v_\ep\right)Ju_\ep-\nabla P\left(\div\big(\curl\left(v_\ep\right)Ju_\ep\big)\right)\right]+P\big[\nabla P\big(\div\left(f\right)\big)-f\big]+L(u_\ep,p_\ep),
\eeq
where $L$ is a $C^\infty$-kernel operator acting on the pair $(u_\ep,p_\ep)$. Using the representation \refe{Duve}. of $\curl\left(v_\ep\right)Ju_\ep$, and setting
\beq\label{ge}
g_\ep:=\Div\big((v_\ep-v)\otimes u_\ep\big)-\nabla\big((v_\ep-v)\cdot u_\ep\big),
\eeq
we get
\beq\label{ue22}
u_\ep=P\left[(Du_\ep)^Tv_\ep+g_\ep-\nabla P\left(\div\big((Du_\ep)^Tv_\ep+g_\ep\big)\right)\right]+F(u_\ep,p_\ep),
\eeq
where
\beq\label{F}
\ba{ll}
F(\zeta,\th):= & P\left[\Div(v\otimes \zeta)-\nabla(v\cdot\zeta)-f-\nabla P\left(\div\big(\Div(v\otimes\zeta)-\nabla(v\cdot\zeta)-f\big)\right)\right]
\\ \ecart
& +\,L(\zeta,\th).
\ea
\eeq
Note that by \refe{regP}. we have
\[
F(u_\ep,p_\ep)\longrightarrow F(u,p)\quad\mbox{strongly in }W^{1,r}_{\rm loc}(\Om),\quad\mbox{for any }r\in(1,2).
\]
Moreover, by \refe{nuS2}. the sequence $(Du_\ep)^Tv_\ep$ weakly converges to $\nu+(Du)^Tv$ in $L^1(\Om)^2$ which is compactly embedded in $W^{-1,r}_{\rm loc}(\Om)^2$ for any $r\in(1,2)$. Hence, as in the first step of the proof of Theorem~\ref{thm1}, from \refe{ue22}. and the two previous convergences we deduce, for any $r\in(1,2)$, the strong convergence
\beq\label{ue23}
\ba{r}
u_\ep-P\left[g_\ep-\nabla P\big(\div\left(g_\ep\right)\big)\right]\longrightarrow
P\left[\nu+(Du)^Tv-\nabla P\left(\div\big(\nu+(Du)^Tv\big)\right)\right]+F(u,p)
\\ \ecart
\mbox{strongly in }W^{1,r}_{\rm loc}(\Om)^2.
\ea
\eeq
\par\bs\noindent
{\it Second step:} Determination of the limit $\si^k$ of $(Du_\ep)^T v_\ep^k$.
\par\ms\noindent
Fix $r\in(1,2)$ such that \refe{ue23}. holds. Set
\beq\label{zeqe}
z_\ep:=P\left[g_\ep-\nabla P\big(\div\left(g_\ep\right)\big)\right]\quad\mbox{and}\quad q_\ep:=P\big(\div\left(g_\ep\right)\big).
\eeq
In view of \refe{ge}. the sequence $g_\ep$ weakly converges to $0$ in $W^{-1,r}(\Om)^2$, hence by \refe{regP}. we have
\beq\label{conzeqe}
z_\ep\confai 0\;\;\mbox{weakly in }W^{1,r}_{\rm loc}(\Om)^2\quad\mbox{and}\quad q_\ep\confai 0\;\;\mbox{weakly in }L^r_{\rm loc}(\Om)/\RR.
\eeq
Moreover, by \refe{DP}. we have
\beq\label{eqzeqe}
\De z_\ep=g_\ep-\nabla q_\ep-K'g_\ep\quad\mbox{and}\quad\De q_\ep=\div\left(g_\ep\right)-K'q_\ep\quad\mbox{in }\Om,
\eeq
hence
\[
\De\big(\div\left(z_\ep\right)\big)=K'q_\ep-\div\left(K'g_\ep\right)\longrightarrow 0\quad\mbox{strongly in }L^r_{\rm loc}(\Om)^2,\mbox{ say}.
\]
This combined with the first convergence of \refe{conzeqe}. and \refe{regP}. yields
\beq\label{divze}
\div\left(z_\ep\right)\longrightarrow 0\quad\mbox{strongly in }W^{2,r}_{\rm loc}(\Om)^2.
\eeq
On the other hand, set $v_\ep^k:=v_\ep\,1_{\{|v_\ep|<k\}}$, for a positive integer $k$. Up to a subsequence of $\ep$ still denoted by $\ep$, $v_\ep^k$ weakly converges to some function $v^k$ in $L^2(\Om)^2$ for any $k$. Consider for $\la\in\RR^2$, the solutions $w_\ep^{\la,k}$ and $q_\ep^{\la,k}$ of the Stokes problem
\beq\label{Sewelk}
\left\{\ba{rll}
-\,\De w_\ep^{\la,k}+\Div\left((v_\ep^k-v^k)\otimes\la\right)+\nabla q_\ep^{\la,k} & =0 & \mbox{in }\Om
\\ \ecart
\div\left(w_\ep^{\la,k}\right) & =0 & \mbox{in }\Om
\\ \ecart
w_\ep^{\la,k} & =0 & \mbox{on }\partial\Om.
\ea\right.
\eeq
which consists in a approximation of equation \refe{Sewel}.. By the regularity results for the Stokes equation (see, e.g., \cite{Lad}) we have
\beq\label{cwelk}
\left\{\ba{rll}
w_\ep^{\la,k} & \confai 0 & \mbox{weakly in }W^{1,s}_{\rm loc}(\Om)^2
\\ \ecart
q_\ep^{\la,k} & \confai 0 & \mbox{weakly in }L^s_{\rm loc}(\Om)/\RR,
\ea\right.
\quad\mbox{for any }s\in(1,\infty).
\eeq
Choose $s:=r'$, and apply the Tartar method (see Appendix of \cite{San}). Let $\ph\in C^\infty_c(\Om)$. Putting $\ph\,w_\ep^{\la,k}$ in the first equation of \refe{eqzeqe}. and $\ph\,z_\ep$ in equation \refe{Sewelk}., and using the definition \refe{ge}. of $g_\ep$ and the convergences \refe{conzeqe}., \refe{divze}., \refe{cwelk}. we have
\[
(Dz_\ep)^T(v_\ep^k-v^k)\cdot\la-\big(Dw_\ep^{\la,k}\big)^T(v_\ep-v)\cdot u_\ep\confai 0\quad\mbox{weakly in }\D'(\Om)^2.
\]
Hence, since $Dz_\ep$ weakly converges to $0$ in $L^r(\Om)^{2\times 2}$, we deduce that
\beq\label{DzeDwelk}
(Dz_\ep)^T v_\ep^k\confai M^{k}u\quad\mbox{weakly in }\D'(\Om)^2,
\eeq
where the matrix-valued function $M^{k}$ is defined by
\beq\label{MTk}
\big(Dw_\ep^{\la,k}\big)^T v_\ep\confai M^{k}\,\la\quad\mbox{weakly in }L^s_{\rm loc}(\Om)^2,\quad\mbox{for any }s\in[1,2).
\eeq
Now, we are able to determine the limit $\si^k$ of the sequence $(Du_\ep)^T v_\ep^k$ in $L^2(\Om)^2$. With the definition \refe{zeqe}. of $z_\ep$ the strong convergence \refe{ue23}. implies that
\[
\ba{r}
(Du_\ep)^T v_\ep^k-(Dz_\ep)^T v_\ep^k\confai\left(DP\left[\nu+(Du)^Tv-\nabla P\left(\div\big(\nu+(Du)^Tv\big)\right)\right]\right)^T v^k+\big(DF(u,p)\big)^T v^k
\\ \ecart
\quad\mbox{weakly in }L^r_{\rm loc}(\Om)^2.
\ea
\]
This combined with \refe{DzeDwelk}. thus yields
\beq\label{sikS}
\si^k=M^{k}u+\left(DP\left[\nu+(Du)^Tv-\nabla P\left(\div\big(\nu+(Du)^Tv\big)\right)\right]\right)^T v^k+\big(DF(u,p)\big)^T v^k.
\eeq
\par\ms\noindent
{\it Third step:} Determination of the limit equation \refe{eqBT}..
\par\ms\noindent
The function  $u$ solves the equation \refe{eqSnu}. which by \refe{Duv}. and similarly to \refe{ue22}., can read as
\[
u=P\left[\nu+(Du)^Tv-\nabla P\left(\div\big(\nu+(Du)^Tv\big)\right)\right]+F(u,p).
\]
This implies that
\[
(Du)^T v^k=\left(DP\left[\nu+(Du)^Tv-\nabla P\left(\div\big(\nu+(Du)^Tv\big)\right)\right]\right)^T v_k+\big(DF(u,p)\big)^T v^k.
\]
Therefore, equating the previous equation with \refe{sikS}. yields
\beq\label{sikMTk}
\si^k=(Du)^T v^k+M^{k}u\quad\mbox{in }\Om.
\eeq
\par
It remains to pass to the limit as $k$ tends to infinity. Due to the equi-integrability of $v_\ep$ in $L^2(\Om)^2$ and by convergence \refe{nuS2}. the sequence $\si^k$ strongly converges to $\nu+(Du)^T v$ in $L^1(\Om)$.
On the other hand, putting the function $w_\ep^{\la,k}-w_\ep^\la$ both in equations \refe{Sewel}. and \refe{Sewelk}. we get the equality
\[
\int_\Om\big|Dw_\ep^{\la,k}-Dw_\ep^\la\big|^2\,dx=\int_\Om\big(Dw_\ep^{\la,k}-Dw_\ep^\la\big)^T(v_\ep^k-v_\ep-v^k+v)\cdot\la\,dx,
\]
which, again by the equi-integrability of $v_\ep$, yields
\beq\label{estDwelk}
\lim_{k\to\infty}\;\sup_{\ep>0}\left(\int_\Om\big|Dw_\ep^{\la,k}-Dw_\ep^\la\big|^2\,dx\right)=0.
\eeq
Estimate \refe{estDwelk}. implies that the sequence $M^{k}$ defined by \refe{MTk}. strongly converges in $L^1(\Om)^{2\times 2}$ to the matrix-valued function ${M}$ defined by \refe{MT}.. In particular, up to a subsequence $M^{k}$ converges to ${M}$ a.e. in $\Om$. Then, by the Fatou lemma combined with \refe{sikMTk}. and the strong convergences of $\si^k$ in $L^1(\Om)^2$ and $v_k$ in $L^2(\Om)^2$, we get that the function ${M}u$ belongs to $L^1(\Om)^2$. Finally, passing to the limit in \refe{sikMTk}. we obtain the equality
\[
\nu={M}u\quad\mbox{in }\Om,
\]
which gives the limit equation \refe{eqBT}..
\par\bs\noindent
{\it Fourth step:} Proof of inequality \refe{MTuu}..
\par\ms\noindent
Similarly to \refe{We}. let $W_\ep^k$, $k>0$, be the matrix-valued function defined by $W_\ep^k\la:=w_\ep^{\la,k}$, where $w_\ep^{\la,k}$ solves \refe{Sewelk}.. We simply denote $w_\ep^{i,k}$ when $\la=e_i:=(2-i,i-1)$, for $i=1,2$. Let $\ph\in C^1_c(\Om)^2$, and let $t\in\RR$.
Using \refe{estDue}. we have
\beq\label{DueDWek}
\ba{l}
\dis\int_\Om\left|Du_\ep-Du-t\,D(W_\ep^k \ph)\right|^2 dx=\langle f,u\rangle_{H^{-1}(\Om)^2,H^1_0(\Om)^2}-\int_\Om|Du|^2\,dx
\\ \ecart
\dis -\,2\,t\int_\Om Du_\ep:D(W_\ep^k \ph)\,dx+t^2\int_\Om\left|D(W_\ep^k\ph)\right|^2 dx+o(1).
\ea
\eeq
Moreover, similarly to the second convergences of \refe{MTr}. and \refe{Dueve}., we have for $i,j=1,2$,
\beq\label{MTkMk}
\left\{\ba{rlll}
Du_\ep:Dw_\ep^{i,k} & \confai M^{k}u\cdot e_i
\\ \ecart
\dis Dw_\ep^{i,k}\cdot Dw_\ep^{j,k} & \confai \hat{M}^k e_i\cdot e_j,
\ea\right.
\quad\mbox{weakly in }L^s_{\rm loc}(\Om),\quad\mbox{for any }s\in[1,2),
\eeq
where (compare to the definition \refe{MTk}. of $M^{k}$) the matrix-valued $\hat{M}^k$ is defined by
\beq\label{Mk}
\big(Dw_\ep^{i,k}\big)^T v_\ep^k\confai \hat{M}^k\,e_i\quad\mbox{weakly in }L^2(\Om)^2.
\eeq
Then, from convergences \refe{cwelk}. and \refe{MTkMk}. we deduce that
\[
\ba{rll}
\dis \int_\Om Du_\ep:D(W_\ep^k\ph)\,dx & \dis =\sum_{i=1}^2\int_\Om Du_\ep:Dw_\ep^{i,k}\,\ph_i\,dx+o(1) & \dis \limo\int_\Om M^{k}u\cdot\ph\,dx,
\\ \ecart
\dis \int_\Om\left|D(W_\ep^k\ph)\right|^2 dx & \dis =\sum_{i,j=1}^2\int_\Om Dw_\ep^{i,k}:Dw_\ep^{j,k}\,\ph_i\,\ph_j\,dx+o(1) & \dis \limo\int_\Om \hat{M}^k\ph\cdot\ph\,dx,
\ea
\]
This combined with \refe{DueDWek}. implies that
\beq\label{DueDWek2}
\ba{l}
\dis\int_\Om\left|Du_\ep-Du-t\,D(W_\ep^k \ph)\right|^2 dx=\langle f,u\rangle_{H^{-1}(\Om)^2,H^1_0(\Om)^2}-\int_\Om|Du|^2\,dx
\\ \ecart
\dis -\,2\,t\int_\Om M^{k}u\cdot\ph\,dx+t^2\int_\Om \hat{M}^k\ph\cdot\ph\,dx+o(1).
\ea
\eeq
Therefore, we have for any $t\in\RR$,
\[
t^2\int_\Om \hat{M}^k\ph\cdot\ph\,dx-2\,t\int_\Om M^{k}u\cdot\ph\,dx+\langle f,u\rangle_{H^{-1}(\Om)^2,H^1_0(\Om)^2}-\int_\Om|Du|^2\,dx\geq 0,
\]
hence
\beq\label{MTkup1}
\left(\int_\Om M^{k}u\cdot\ph\,dx\right)^2\leq\left(\langle f,u\rangle_{H^{-1}(\Om)^2,H^1_0(\Om)^2}-\int_\Om|Du|^2\,dx\right)\int_\Om \hat{M}^k\ph\cdot\ph\,dx.
\eeq
Let $\de>0$, and let $\om$ be an open set such that $\om\Subset\Om$. Since by \refe{MTk}. and \refe{Mk}. $M^{k}$ and $\hat{M}^k$ belong to $L^s(\om)^{2\times 2}$ for $s\in[1,2)$, putting in \refe{MTkup1}. strong approximations $\ph$ of ${1_\om\,u\over 1+\de\,|u|}$ in $L^{2s'}(\Om)^2$, we get
\[
\ba{ll}
\dis \left(\int_\om{M^{k}u\cdot u\over 1+\de\,|u|}\,dx\right)^2 &
\dis \leq\left(\langle f,u\rangle_{H^{-1}(\Om)^2,H^1_0(\Om)^2}-\int_\Om|Du|^2\,dx\right)\int_\om{\hat{M}^{k}u\cdot u\over(1+\de\,|u|)^2}\,dx
\\ \ecart
& \dis \leq\left(\langle f,u\rangle_{H^{-1}(\Om)^2,H^1_0(\Om)^2}-\int_\Om|Du|^2\,dx\right)\int_\Om{\hat{M}^{k}u\cdot u\over(1+\de\,|u|)^2}\,dx,
\ea
\] 
which by the arbitrariness of $\om$ yields the inequality
\beq\label{MTkup2}
\left(\int_\Om{M^{k}u\cdot u\over 1+\de\,|u|}\,dx\right)^2
\leq\left(\langle f,u\rangle_{H^{-1}(\Om)^2,H^1_0(\Om)^2}-\int_\Om|Du|^2\,dx\right)\int_\Om{\hat{M}^{k}u\cdot u\over(1+\de\,|u|)^2}\,dx.
\eeq
Recall that, by virtue of the equi-integrability of $v_\ep$ in $L^2(\Om)^2$, the sequences $M^{k}$ and $\hat{M}^k$ strongly converge to ${M}$ in $L^1(\Om)^{2\times 2}$, thus converge, up to a subsequence of $k$, a.e. in $\Om$ and in a dominated way. Therefore, passing to the limit as $k\to\infty$ owing to the Fatou lemma for the left-hand side of \refe{MTkup2}. and owing to the Lebesgue dominated convergence theorem for the right-hand side of \refe{MTkup2}., it follows that
\[
\left(\int_\Om{{M}u\cdot u\over 1+\de\,|u|}\,dx\right)^2
\leq\left(\langle f,u\rangle_{H^{-1}(\Om)^2,H^1_0(\Om)^2}-\int_\Om|Du|^2\,dx\right)\int_\Om{{M}u\cdot u\over(1+\de\,|u|)^2}\,dx<\infty,
\]
which implies the inequality
\beq\label{MTkup3}
\int_\Om{{M}u\cdot u\over 1+\de\,|u|}\,dx\leq\langle f,u\rangle_{H^{-1}(\Om)^2,H^1_0(\Om)^2}-\int_\Om|Du|^2\,dx.
\eeq
Finally, applying the Fatou lemma in \refe{MTkup3}. as $\de\to 0$  we obtain the desired inequality \refe{MTuu}..
\par\bs\noindent
{\it Fifth step:} Proof of equality \refe{MTuu=}. and of the corrector result \refe{corueWe}..
\par\ms\noindent
Assume that $\Om$ has a Lipschitz boundary, $v\in L^r(\Om)^N$, with $r>2$, and ${M}\in L^m(\Om)^{2\times 2}$, with $m>1$.
Let $\ph$ be a divergence free function in $C^\infty_c(\Om)^2$. Putting $\ph$ as test function in the limit Stokes equation \refe{eqBT}. and using the representation formula \refe{Duv}. we have
\beq\label{vfeqBT}
\ba{l}
\dis \int_\Om Du:D\ph\,dx+\int_\Om(Du)^T v\cdot\ph\,dx-\int_\Om(v\otimes u):D\ph\,dx+\int_\Om {M}u\cdot\ph\,dx
\\ \ecart
\dis =\langle f,u\rangle_{H^{-1}(\Om)^2,H^1_0(\Om)^2}.
\ea
\eeq
Due to the regularity of $\Om$ the set of divergence free functions is known to be dense in the space of divergence free functions in $H^1_0(\Om)^2$ (see, e.g., \cite{Tar5}).
Moreover, by the higher integrability of $v$ and ${M}$ the mapping
\[
\ph\longmapsto\int_\Om(Du)^T v\cdot\ph\,dx-\int_\Om(v\otimes u):D\ph\,dx+\int_\Om {M}u\cdot\ph\,dx
\]
is continuous in $H^1_0(\Om)^2$.
Therefore, considering in \refe{vfeqBT}. a divergence free strong approximation $\ph$ of $u$ in $H^1_0(\Om)^2$ we get
\[
\int_\Om|Du|^2\,dx+\int_\Om(Du)^T v\cdot u\,dx-\int_\Om(v\otimes u):Du\,dx+\int_\Om {M}u\cdot u\,dx=\langle f,u\rangle_{H^{-1}(\Om)^2,H^1_0(\Om)^2}\,,
\]
which is \refe{MTuu=}.. This equality clearly implies the uniqueness of a solution $u\in H^1_0(\Om)^2$ of \refe{eqBT}., with ${M}u\cdot u\in L^1(\Om)$.
\par
It remains to prove the corrector result \refe{corueWe}.. Let $\ph\in C^\infty_c(\Om)$. Applying successively the triangle inequality and the Cauchy-Schwarz inequality we have
\[
\ba{l}
\dis \int_\Om\big|Du_\ep-Du-D(W_\ep\,u)\big|\,dx
\\ \ecart
\dis \leq\int_\Om\big|Du_\ep-Du-D(W_\ep\,\ph)\big|\,dx+\int_\Om\left|D\big(W_\ep\,(u-\ph)\big)\right|dx
\\ \ecart
\dis \leq\int_\Om\big|Du_\ep-Du-D(W_\ep\,\ph)\big|\,dx+\int_\Om|DW_\ep|\,|u-\ph|\,dx+\int_\Om|W_\ep|\,|Du-D\ph|\,dx
\\ \ecart
\dis \leq\int_\Om\big|Du_\ep-Du-D(W_\ep\,\ph)\big|\,dx+c\left\|W_\ep\right\|_{H^1(\Om)^{2\times 2}}\,\|u-\ph\|_{H^1_0(\Om)^2}\,,
\ea
\]
hence by the boundedness of $W_\ep$ in $H^1_0(\Om)^2$,
\beq\label{DueDWe}
\int_\Om\big|Du_\ep-Du-D(W_\ep\,u)\big|\,dx\leq\int_\Om\big|Du_\ep-Du-D(W_\ep\,\ph)\big|\,dx+c\,\|u-\ph\|_{H^1_0(\Om)^2}.
\eeq
On the other hand, proceeding as in fourth step owing to the second convergences of \refe{MT}. and~\refe{Dueve}. (which hold in the weak-$*$ sense of measures on $\Om$) we get similarly to \refe{DueDWek2}. the equality
\[
\ba{l}
\dis\int_\Om\big|Du_\ep-Du-D(W_\ep\,\ph)\big|^2\,dx=\langle f,u\rangle_{H^{-1}(\Om)^2,H^1_0(\Om)^2}-\int_\Om|Du|^2\,dx
\\ \ecart
\dis -\,2\int_\Om {M}u\cdot\ph\,dx+\int_\Om {M}\,\ph\cdot\ph\,dx+o(1).
\ea
\]
Hence, taking into account equality \refe{MTuu=}. and using the H\"older inequality combined with the embedding of $H^1_0(\Om)$ in any $L^s(\Om)$ space, it follows that
\beq\label{DueDWeph}
\ba{ll}
\dis \int_\Om\big|Du_\ep-Du-D(W_\ep\,\ph)\big|^2\,dx & \dis =\int_\Om {M}\,(u-\ph)\cdot(u-\ph)\,dx+o(1)
\\ \ecart
& \dis \leq c\,\|{M}\|_{L^m(\Om)^{2\times 2}}\,\|u-\ph\|^2_{H^1_0(\Om)^2}+o(1).
\ea
\eeq
Therefore, by \refe{DueDWe}. and \refe{DueDWeph}. we obtain the inequality
\beq\label{corueWe2}
\limsup_{\ep\to 0}\int_\Om\big|Du_\ep-Du-D(W_\ep\,u)\big|\,dx\leq c\,\|u-\ph\|_{H^1_0(\Om)^2}\,,\quad\mbox{for any }\ph\in C^\infty_c(\Om)^2,
\eeq
which implies the desired convergence \refe{corueWe}. and concludes the proof of Theorem~\ref{thm2}. \cqfd
\par\bs
As in the scalar case we show in the next section that the equi-integrability condition is crucial to derive the limit Brinkman equation \refe{eqBT}. with the matrix-valued function ${M}$ introduced by L.~Tartar \cite{Tar2,Tar4}.
\subsection{A counter-example}\label{s.cexS}
Let $\Om$ be a regular bounded domain of $\RR^2$. For $\ep>0$, let $\om_\ep$ be the intersection of $\Om$ with the periodic lattice of disks of center $2\ep\,\ka$, $\ka\in\ZZ^2$, and of radius $\ep\,r_\ep$ such that
\beq\label{ga}
{4\pi\over\ep^2\,|\ln r_\ep|}\limo\ga\in(0,\infty).
\eeq
This geometry was used by Cioranescu, Murat~\cite{CiMu} for the Laplace equation and by Allaire \cite{All} for the Stokes equation, in order to derive a ``strange term" of zero-order from the homogenization of the Dirichlet boundary conditions on the small disks.
\par
In the square $Y:=(-1,1)^2$, let $Q$ be the disk centered at the origin and of radius $1$, and let $Q_{r_\ep}$ be the disk of same center and of radius $r_\ep$ with measure $|Q_{r_\ep}|=\pi\,r_\ep^2$. Then, for $f\in H^{-1}(\Om)^2$, we consider the Stokes equation
\beq\label{pSeoe}
\left\{\ba{rll}
\dis -\,\De u_\ep+{1_{\om_\ep}\over|Q_{r_\ep}|}\,Ju_\ep+\nabla p_\ep & =f & \mbox{in }\Om
\\
\div\left(u_\ep\right) & =0 &  \mbox{in }\Om
\\ \ecart
u_\ep & =0 &  \mbox{on }\partial\Om.
\ea\right.
\eeq
Note that, in view of the definition of $\om_\ep$, we have $|\om_\ep|\approx|\Om|\,|Q_{r_\ep}|$. Moreover, if $z_\ep\in H^1_0(\Om)$ is the solution of the Laplace equation
\beq\label{ze}
\De z_\ep={1_{\om_\ep}\over|Q_{r_\ep}|}\quad\mbox{in }\D'(\Om),
\eeq
we have
\beq\label{vecex}
{1_{\om_\ep}\over|Q_{r_\ep}|}=\curl\left(v_\ep\right)\;\;\mbox{in }\D'(\Om),\quad\mbox{where}\quad v_\ep:=J\nabla z_\ep.
\eeq
Hence, the Stokes problem \refe{pSeoe}. is of the same type as \refe{pSe}..
On the other hand, using successively the Cauchy-Schwarz inequality and the estimate \refe{eoe}. below combined with \refe{ga}. we have
\[
\int_\Om|\nabla z_\ep|^2\,dx=-\,{|\om_\ep|\over|Q_{r_\ep}|}\,\moy_{\om_\ep}z_\ep\,dx\leq{|\om_\ep|\over|Q_{r_\ep}|}\left(\moy_{\om_\ep}z_\ep^2\,dx\right)^{1\over 2}
\leq c\,\|\nabla z_\ep\|_{L^2(\Om)^2},
\]
which implies that $z_\ep$ is bounded in $H^1_0(\Om)$. Therefore, the sequence $v_\ep$ is bounded in $L^2(\Om)^2$.
Moreover, since by periodicity the sequence ${1_{\om_\ep}\over|Q_{r_\ep}|}$ converges weakly-$*$ to ${1\over 4}$ in $\M(\Om)$, we get
\beq\label{convecex}
v_\ep\confai v\;\;\mbox{weakly in }L^2(\Om)^2,\quad\mbox{with}\quad\curl\left(v\right)={1\over 4}\;\;\mbox{in }\D'(\Om).
\eeq
On the other hand, it is not difficult to check that $v_\ep$ is not equi-integrable in $L^2(\Om)^2$. In fact, the following result shows that \refT{thm2}. does not hold for this particular sequence $v_\ep$:
\begin{Thm}\label{thm.cexS}
The sequence $u_\ep$ weakly converges in $H^1_0(\Om)^2$ to the solution $u$ of the Brinkman equation
\beq\label{Be}
\left\{\ba{rll}
\dis -\,\De u+{1\over 4}\,Ju+\nabla p+\Ga{u} & =f & \mbox{in }\Om
\\ \ecart
\div\left(u\right) & =0 &  \mbox{in }\Om
\\ \ecart
u & =0 &  \mbox{on }\partial\Om,
\ea\right.
\eeq
where the extra zero-order term $\Ga{u}$ is given by
\beq\label{Gau}
\left({1_{\om_\ep}\over|Q_{r_\ep}|}-{1\over 4}\right)Ju_\ep\confai \Ga{u}\quad\mbox{weakly-$*$ in }\M(\Om)^2,
\eeq
and $\Ga$ is the constant matrix defined by
\beq\label{Ga}
\Ga:={1\over 4\left(\ga^2+1\right)}\left(\ga\,I-J\right).
\eeq
Moreover, the matrix obtained from convergence \refe{MT}. according to the Tartar approach is given by
\beq\label{MTcex}
{M}={1\over 4\,\ga}\,I.
\eeq
\end{Thm}
\begin{Rem}\label{remB}
The matrix $\Ga$ of the Brinkman equation \refe{Be}. is not symmetric contrary to the matrix ${M}$ arising in the Tartar approach.
Moreover, we have
\[
\Ga{u}\cdot u<{M}u\cdot u\quad\mbox{if }u\neq 0.
\]
The gap between the two previous energies (which are the energies dissipated by viscosity according to \cite{Tar1}) is due to the loss of equi-integrability of the sequence $v_\ep$ defined by \refe{vecex}.. Therefore, the equi-integrability of $v_\ep$ can be regarded as an optimal condition to ensure the result of Theorem~\ref{thm2}.
\end{Rem}
\begin{Rem}\label{rem.cor}
It is worth to mention that the pathology displayed in Theorem~\ref{thm.cexS} is not due to the absence of correctors. Indeed, with the oscillating sequences $v^1_\ep$, $v^2_\ep$ defined by \refe{vpie}., \refe{VPie}. below, the following corrector result holds:
\end{Rem}
\begin{Pro}\label{pro.cor}
Assume that $u\in W^{1,r}(\Om)^2$ for some $r>2$. Then, we have
\beq\label{scue}
\ba{ll}
& u_\ep-u-v_1\,v_\ep^1-v_2\,v_\ep^2\;\longrightarrow\; 0\quad\mbox{strongly in }H^1(\Om),
\\ \ecart
\mbox{where} & \dis v=(v_1,v_2):={1\over\ga^2+1}\left(-\,u_1+\ga u_2,-\,u_2-\ga u_1\right).
\ea
\eeq
\end{Pro}
\begin{Rem} If the right-hand side $f$ belongs to $W^{-1,r}(\Om)^2$ for some $r>2$, then using the regularity results for the Stokes equation (see, e.g., \cite{Lad}) the solution $u$ of the Stokes equation~\refe{Be}. belongs to $W^{1,r}(\Om)^2$. This provides a quite general condition under which the strong convergence \refe{scue}. holds. 
\end{Rem}
\par
The proof of \refT{thm.cexS}. is partially based on the properties of the test functions $v_\ep^1$, $v_\ep^2$ defined by \refe{vpie}., \refe{VPie}. below, and introduced by Allaire \cite{All}. They were also used in \cite{Bri} to derive a homogenized Brinkman type equation but, contrary to \refe{pSe}., from a Stokes equation without zero-order term. More precisely, in \cite{All} the velocity is assumed to be zero in the set~$\om_\ep$. In~\cite{Bri} the viscosity is assumed to be very high in cylinders of section $\om_\ep$, which leads to a three-dimensional nonlocal Brinkman equation.
In the perturbed Stokes equation \refe{pSeoe}. a highly oscillating zero-order term is concentrated on $\om_\ep$.
\par\bs
On the one hand, the sets $Q_{r_\ep}$ and $\om_\ep$ satisfy the following estimates:
\begin{Lem}\label{lem.oe}
There exists a constant $C>0$ such that
\beq\label{eQe}
\forall\,V\in H^1(Y),\quad\left|\,\moy_{Q_{r_\ep}}V\,dy-\moy_{Y}V\,dy\,\right|\leq C\,\sqrt{|\ln r_\ep|}\,\|\nabla V\|_{L^2(Y)^2},
\eeq
\beq\label{eoe}
\forall\,v\in H^1_0(\Om),\quad\moy_{\om_\ep}|v|^2\,dx\leq C\left(1+\ep^2|\ln r_\ep|\right)\|\nabla v\|^2_{L^2(\Om)^2}.
\eeq
\end{Lem}
\begin{proof}
Estimate~\refe{eQe}. can be easily proved using the polar coordinates. Estimate \refe{eoe}. is an immediate consequence of the Lemma~3 of \cite{PiSe}, and can also be deduced from \refe{eQe}..
\end{proof}
\par
On the other hand, consider the $\ep Y$-periodic functions $v_\ep^i$ and $p_\ep^i$, for $i=1,2$, defined by
\beq\label{vpie}
v_\ep^i(x):=V_\ep^i\left({x\over\ep}\right),\quad p_\ep^i(x):={1\over\ep}\,P_\ep^i\left({x\over\ep}\right),\quad\mbox{for }x\in\RR^2,
\eeq
where $V_\ep^i\in H^1_\#(Y)$ are $P_\ep^i\in L^2(Y)$ are the $Y$-periodic functions defined by
\beq\label{VPie}
V_\ep^i:=\left\{\ba{ll} e_i & \mbox{in }Q_{r_\ep} \\ 0 & \mbox{in }Y\setminus Q,\ea\right.
\quad P_\ep^i=0\;\;\mbox{in }Q_{r_\ep}\cup\left(Y\setminus Q\right),\quad\int_Y P_\ep^i\,dy=0,
\eeq
which solve the Stokes equation
\beq\label{SVie}
-\,\De V_\ep^i+\nabla P_\ep^i=0\quad\mbox{in }Q_{r_\ep}\setminus\bar{Q}.
\eeq
Moreover, the sequences $V_\ep^i$ and $P_\ep^i$ satisfy the following estimates:
\begin{Lem}\label{lem.veil}
There exists a constant $C>0$ such that
\beq\label{eVPie}
\left\{\ba{rl}
\dis \|V_\ep^i\|_{L^2(Y)^2}+\|DV_\ep^i\|^2_{L^2(Y)^{2\times 2}}+\|P_\ep^i\|^2_{L^2(Y)} & \dis \leq{C\over|\ln r_\ep|}
\\ \ecart
\dis \|V_\ep^i\|_{L^\infty(Y)^2} & \leq C,
\ea\right.
\eeq
and for any function $V\in H^1(Y)$,
\beq\label{eDVie}
\ba{l}
\dis \left|\,\int_Y DV_\ep^i:DV\,dy-\int_Y P_\ep^i\,\div\left(V\right)dy-\ga_\ep^i\,e_i\cdot\left(\moy_{Q_{r_\ep}}V-\moy_{Y\setminus Q}V\right)\right|
\\ \ecart
\dis \leq{C\over|\ln r_\ep|}\,\|DV\|_{L^2(Y)^{2\times 2}},
\ea
\eeq
where $-\kern-.85em\int$ denotes the average value and 
\beq\label{gie}
\ga_\ep^i\;\mathop{\approx}_{\ep\to 0}\;{4\pi\over|\ln r_\ep|}.
\eeq
\end{Lem}
\begin{proof}
Estimate \refe{eVPie}. can be proved using the polar coordinates (see also \cite{All}). Estimate \refe{eDVie}. is a straightforward consequence of the Lemma~3.3 of \cite{Bri} (with a refinement for the right-hand side of the inequality).
\end{proof}
\par\noindent
{\bf Proof of Theorem~\ref{thm.cexS}.}
The proof is divided into two steps. In the first step we determine the homogenized Brinkman equation \refe{Be}.. The second step is devoted to the computation of the matrix ${M}$ defined in the Tartar approach.
\par\ms\noindent
{\it First step: } Determination of the homogenized equation.
\par\ms\noindent
Using $u_\ep$ as test function we have
\[
\int_\Om|Du_\ep|^2=\langle f, u_\ep\rangle_{H^{-1}(\Om)^2,H^1_0(\Om)^2}\leq c\,\|f\|_{H^{-1}(\Om)^2}\,\|Du_\ep\|_{L^2(\Om)^{2\times 2}}\,,
\]
which implies that $u_\ep$ is bounded in $H^1_0(\Om)^2$.
On the other hand, let $\ph\in C^\infty_c(\Om)$ with zero $\Om$-average. There exists (see, e.g., \cite{Bog}) a vector-valued function $\Phi\in C^\infty_c(\Om)^2$ such that
\[
\div\left(\Phi\right)=\ph\;\;\mbox{in }\Om\quad\mbox{and}\quad\|\Phi\|_{H^1_0(\Om)^2}\leq c\,\|\ph\|_{L^2(\Om)}\,,
\]
where the constant $c$ is independent of $\ph$, $\Phi$. Using $\Phi$ as test function in equation \refe{pSeoe}. and applying successively the Cauchy-Schwarz inequality, estimates \refe{eoe}. and \refe{ga}. we get
\[
\ba{ll}
\dis \left|\,\int_\Om p_\ep\,\ph\,dx\,\right| & \dis \leq\left|\,\langle f,\Phi\rangle_{H^{-1}(\Om)^2,H^1_0(\Om)^2}\,\right|
+\left|\,\int_\Om Du_\ep:D\Phi\,\right|+\left|\,\int_\Om{1_{\om_\ep}\over|Q_{r_\ep}|}\,Ju_\ep\cdot\Phi\,\right|
\\ \ecart
& \dis \leq c\,\|D\Phi\|_{L^2(\Om)^{2\times 2}}+c\left(\moy_{\om_\ep}|u_\ep|^2\right)^{1\over 2}\left(\moy_{\om_\ep}|\Phi|^2\right)^{1\over 2}
\\ \ecart
& \dis \leq c\,\|D\Phi\|_{L^2(\Om)^{2\times 2}}+c\,\ep^2\,|\ln r_\ep|\,\|Du_\ep\|_{L^2(\Om)^{2\times 2}}\,\|D\Phi\|_{L^2(\Om)^{2\times 2}}
\leq c\,\|\ph\|_{L^2(\Om)}.
\ea
\]
This combined with the regularity of $\Om$ implies that $p_\ep$ is bounded in $L^2(\Om)/\RR$.
Therefore, up to a subsequence the following convergences hold
\beq\label{uepecex}
\left\{\ba{rll}
u_\ep & \confai u & \mbox{weakly in }H^1_0(\Om)^2
\\ \ecart
p_\ep & \confai p & \mbox{weakly in }L^2(\Om)/\RR.
\ea\right.
\eeq
\par
Now, we have to determine the limit of the sequence ${1_{\om_\ep}\over|Q_{r_\ep}|}\,Ju_\ep$.
On the one hand, re-scaling inequality \refe{eDVie}. we obtain that the functions $v_\ep^i$ and $p_\ep^i$, $i=1,2$, of \refe{vpie}. and any function $v\in H^1_0(\Om)^2$ satisfy the inequality
\beq\label{eDvie}
\ba{l}
\dis \left|\,\int_\Om Dv_\ep^i:Dv-\int_\Om p_\ep^i\,\div\left(v\right)-{\ga_\ep^i\over\ep^2}\,e_i\cdot\left(\int_\Om{1_{\om_\ep}\over|Q_{r_\ep}|}\,v
-\int_\Om{1_{Y\setminus Q}\over|Y\setminus Q|}\left({x\over\ep}\right)v\right)\right|
\\ \ecart
\dis \leq{c\over\ep\,|\ln r_\ep|}\,\|Dv\|_{L^2(\Om)^{2\times 2}}.
\ea
\eeq
Moreover, by \refe{eVPie}. and \refe{ga}. the following convergences hold
\beq\label{cvpie}
\left\{\ba{rll}
v_\ep^i & \confai 0 & \mbox{weakly in }H^1(\Om)^2
\\ \ecart
p_\ep^i & \confai 0 & \mbox{weakly in }L^2(\Om)/\RR.
\ea\right.
\eeq
Then, applying inequality \refe{eDvie}. with $v=\ph\,u_\ep$, $\ph\in C^\infty_c(\Om)$, we deduce from \refe{gie}. and \refe{ga}. that
\beq\label{e1t2}
\int_\Om Dv_\ep^i:Du_\ep\,\ph-\left(\ga+o(1)\right)e_i\cdot\left(\int_\Om{1_{\om_\ep}\over|Q_{r_\ep}|}\,\ph\,u_\ep-\int_\Om{1\over 4}\,\ph\,u\right)=o(1).
\eeq
On the other hand, putting $\ph\,v_\ep^i$ as test function in \refe{pSeoe}., using that $v_\ep^i=e_i$ in $\om_\ep$ and the convergences \refe{cvpie}., \refe{uepecex}., we have
\beq\label{e2t2}
\int_\Om Du_\ep:Dv_\ep^i\,\ph+\int_\Om{1_{\om_\ep}\over|Q_{r_\ep}|}\,Ju_\ep\cdot e_i\,\ph=o(1).
\eeq
Denote
\[
\nu:=\lim_{\ep\to 0}\,{1_{\om_\ep}\over|Q_{r_\ep}|}\,Ju_\ep\quad\mbox{weakly-$*$ in }\M(\Om)^2,
\]
where the limit holds up to a subsequence by virtue of the estimate \refe{eoe}. combined with the Cauchy-Schwarz inequality.
Then, equating \refe{e1t2}. and \refe{e2t2}. and passing to the limit we get for $i=1,2$,
\[
\int_\Om\ph\,e_i\cdot\nu=\ga\int_\Om\ph\,e_i\cdot J\nu+{\ga\over 4}\int_\Om\ph\,e_i\cdot u,\quad\mbox{for any }\ph\in C^\infty_c(\Om),
\]
which implies the equality $\dis\nu=\ga\,J\nu+{\ga\over 4}\,u$. Hence, we deduce the convergence
\beq\label{nucexS}
{1_{\om_\ep}\over|Q_{r_\ep}|}\,Ju_\ep\confai\nu={\ga\over 4}\,(I-\ga\,J)^{-1}\,u={\ga\over 4\,(\ga^2+1)}\,(I+\ga\,J)\,u\quad\mbox{weakly-$*$ in }\M(\Om)^2.
\eeq
Therefore, passing to the limit in \refe{pSeoe}. with \refe{nucexS}. we obtain the homogenized equation
\beq\label{hBe}
-\,\De u+{1\over 4}\,Ju+\nabla p+{1\over 4\,(\ga^2+1)}\,(\ga\,I-J)\,u=f\quad\mbox{in }\D'(\Om),
\eeq
which yields the desired Brinkman equation \refe{Be}. with the matrix $\Ga$ of \refe{Ga}..
\par\ms\noindent
{\it Second step:} Derivation of the matrix ${M}$.
\par\ms\noindent
Let $\la\in\RR^2$. Consider the solutions $\Wp\in H^1_\sharp(Y)$ (the set of the $Y$-periodic functions in $H^1_{\rm loc}(\RR^2)$) and $\Qp\in L^2_\sharp(Y)/\RR$ of the perturbed Stokes problem
\beq\label{SeWelp}
\left\{\ba{rll}
\dis -\,\De\Wp+\ep\left({1_{Q_{r_\ep}}\over|Q_{r_\ep}|}-{1\over 4}\right)J\la+\nabla\Qp & =0 & \mbox{in }\RR^2
\\
\div\left(\Wp\right) & =0 &  \mbox{in }\RR^2
\\ \ecart
\Wp & \mbox{is} & \mbox{$Y$-periodic}
\\ \ecart
\dis\int_Y\Wp & =0. &
\ea\right.
\eeq
Note that the first equation of \refe{SeWelp}. is equivalent to the variational formulation in the torus,
\beq\label{SeWelpt}
\forall\,V\in H^1_\sharp(Y),\quad\int_Y D\Wp:DV\,dy+\ep\left(\moy_{Q_{r_\ep}}\kern -.4em V-\moy_Y V\right)\cdot J\la-\int_Y\Qp\,\div\left(V\right)dy=0.
\eeq
Hence, the re-scaled functions $\wp$ and $\qp$ defined by
\beq\label{wqelp}
\wp(x):=\ep\,\Wp\left({x\over\ep}\right)\quad\mbox{and}\quad\qp(x):=\Qp\left({x\over\ep}\right),\quad\mbox{for }x\in\Om,
\eeq
are $\ep Y$-periodic solutions of the problem
\beq\label{Sewelp}
\left\{\ba{rll}
\dis -\,\De\wp+\left({1_{Q_{r_\ep}}\over|Q_{r_\ep}|}\left({x\over\ep}\right)-{1\over 4}\right)J\la+\nabla\qp & =0 & \mbox{in }\RR^2
\\
\div\left(\wp\right) & =0 &  \mbox{in }\RR^2.
\ea\right.
\eeq
\par
First of all, let us determine a priori estimates satisfied by the sequences $\Wp$, $\wp$, $\Qp$, and~$\qp$. Putting $\Wp$ as test function in equation \refe{SeWelpt}. we have
\beq\label{e3t2}
\int_Y\big|D\Wp\big|^2\,dy+\ep\moy_{Q_{r_\ep}}J\la\cdot\Wp\,dy=0,
\eeq
hence by the estimates \refe{eQe}. of \refL{lem.oe}. and \refe{ga}.
\[
\big\|D\Wp\big\|^2_{L^2(Y)^{2\times 2}}=\ep\left|\,\moy_{Q_{r_\ep}}\Wp\,dy\,\right|\leq C\,\ep\sqrt{|\ln r_\ep|}\,\big\|D\Wp\big\|_{L^2(Y)^{2\times 2}}
\leq c\,\big\|D\Wp\big\|_{L^2(Y)^{2\times 2}}.
\]
Therefore, $\Wp$ is bounded in $H^1_\sharp(Y)^2$, and there exists a constant vector $\bar{W}^\la\in\RR^2$ such that up to a subsequence we have
\beq\label{bWl}
\lim_{\ep\to 0}\left(\ep\moy_{Q_{r_\ep}}\kern -.4em\Wp\,dy\right)=\bar{W}^\la.
\eeq 
On the other hand, let $\ph\in C^\infty_\sharp(Y)$ with zero $Y$-average. There exists $\Phi\in C^\infty_\sharp(Y)^2$ with zero $Y$-average such that
\[
\div\left(\Phi\right)=\phÖ\;\;\mbox{in }\RR^2\quad\mbox{and}\quad\|\Phi\|_{L^2(Y)^{2\times 2}}\leq c\,\|\ph\|_{L^2(Y)},
\]
where $c$ is a constant independent of $\ph$, $\Phi$. Putting $\Phi$ as test function in \refe{SeWelpt}. we have by \refe{eQe}. and \refe{ga}.
\[
\ba{ll}
\dis \left|\,\int_Y\Qp\,\ph\,dy\,\right| & \dis \leq\left|\,\int_Y D\Wp:D\Phi\,dy\,\right|+\ep\left|\,\moy_{Q_{r_\ep}}\Phi\,dy\,\right|
\\ \ecart
& \dis \leq c\,\|D\Phi\|_{L^2(Y)^{2\times 2}}+c\,\ep\sqrt{|\ln r_\ep|}\,\|D\Phi\|_{L^2(Y)^{2\times 2}}
\\ \ecart
& \dis \leq c\,\|D\Phi\|_{L^2(Y)^{2\times 2}}\leq c\,\|\ph\|_{L^2(Y)},
\ea
\]
hence $\Qp$ is bounded in $L^2_\sharp(Y)/\RR$. From the boundedness and the $Y$-periodicity of $\Wp$ and $\Qp$ we thus deduce that the sequences $\wp$ and $\qp$ of \refe{wqelp}. satisfy the convergences
\beq\label{cwqelp}
\left\{\ba{rll}
\wp & \confai 0 & \mbox{weakly in }H^1(\Om)^2
\\ \ecart
\qp & \confai 0 & \mbox{weakly in }L^2(\Om)/\RR.
\ea\right.
\eeq
\par
Now, let us check that the periodic function $\wp$ of \refe{Sewelp}. gives the same matrix ${M}$ \refe{MT}. as the function $w_\ep^\la$ of \refe{Sewel}. which satisfies a Dirichlet boundary condition. Since ${M}$ is symmetric, this is equivalent to prove that for any $\la\in\RR^2$, 
\beq\label{welwp}
(D\wp)^T v_\ep\cdot\la-(Dw_\ep^\la)^T v_\ep\cdot\la\confai 0\quad\mbox{in }\D'(\Om),
\eeq
where $v_\ep$ is defined by \refe{vecex}.. Let $\ph\in C^\infty_c(\Om)$. Putting $\ph\,\wp$ in the equation \refe{Sewel}. satisfied by $w_\ep^\la$ and $\ph\,w_\ep^\la$ in the equation \refe{Sewelp}. satisfied by $\wp$, and using the convergences \refe{cwqelp}. satisfied by $\wp$, $\qp$ as well as the similar ones satisfied by $w_\ep^\la$, $q_\ep^\la$, we get
\beq\label{e4t2}
\left\{\ba{ll}
\dis \int_\Om Dw_\ep^\la:D\wp\,\ph-\int_\Om\curl\left(v_\ep\right)J\wp\cdot\la\,\ph & \dis \limo 0
\\ \ecart
\dis \int_\Om D\wp:Dw_\ep^\la\,\ph-\int_\Om\curl\left(v_\ep\right)Jw_\ep^\la\cdot\la\,\ph & \dis \limo 0.
\ea\right.
\eeq
Moreover, by the representation formula \refe{Duve}. we have
\beq\label{e5t2}
\left\{\ba{ll}
\curl\left(v_\ep\right)J\wp-(D\wp)^T\,v_\ep & \confai 0
\\ \ecart
\curl\left(v_\ep\right)Jw_\ep^\la-(Dw_\ep^\la)^T\,v_\ep & \confai 0
\ea\right.
\quad\mbox{in }\D'(\Om)^2.
\eeq
Therefore, combining \refe{e4t2}. and \refe{e5t2}. we obtain the desired convergence \refe{welwp}..
\par\medskip
It remains to determine the matrix ${M}$. On the one side putting $\wp$ as test function in \refe{Sewelp}. and using the convergences \refe{welwp}., \refe{MT}., and on the other side using the $\ep Y$-periodicity of $\wp$ \refe{wqelp}., we get similarly to \refe{MTr}. and up to a subsequence
\beq\label{MTp}
\left|D\wp\right|^2\confai {M}\la\cdot\la\;\;\mbox{and}\;\;\left|D\wp\right|^2\confai \lim_{\ep\to 0}\left(\moy_Y\left|D\Wp\right|^2\,dy\right)
\quad\mbox{weakly-$*$ in }\M(\Om).
\eeq
This combined with \refe{e3t2}. and \refe{bWl}. gives
\beq\label{MTbWl}
{M}\la\cdot\la={1\over 4}\,J\bar{W}^\la\cdot\la.
\eeq
Let us compute the constant vector $\bar{W}^\la$. To this end, putting the divergence free function $\Wp$ in the inequality \refe{eDVie}. satisfied by $V_\ep^i$, $i=1,2$, and taking into account the estimates \refe{eVPie}., \refe{ga}. and the boundedness of $\Wp$ in $H^1(Y)^2$, we have
\beq\label{eDVieWpl}
\int_Y DV_\ep^i:D\Wp\,dy=\ga_\ep^i\,e_i\cdot\left(\moy_{Q_{r_\ep}}\Wp\,dy\right)+o(\ep).
\eeq
Moreover, putting the divergence free function $V_\ep^i$ in \refe{SeWelpt}. with $V_\ep^i=e_i$ in $Q_{r_\ep}$, we get
\beq\label{eDVieWpl2}
\int_Y D\Wp:DV_\ep^i\,dy=-\,\ep\left(\moy_{Q_{r_\ep}}V_\ep^i-\moy_{Y}V_\ep^i\right)\cdot J\la=\ep\,Je_i\cdot\la+o(\ep),
\eeq
since by \refe{eVPie}. $V_\ep^i$ strongly converges to zero in $L^2(Y)^2$. The estimates \refe{eDVieWpl}. and \refe{eDVieWpl2}. divided by $\ep$ together with \refe{bWl}., \refe{gie}. and \refe{ga}. imply that
\beq\label{bWl=}
\ga\,e_i\cdot\bar{W}^\la=Je_i\cdot\la\quad\mbox{or equivalently}\quad \bar{W}^\la=-\,{1\over\ga}J\la.
\eeq
This combined with \refe{MTbWl}. yields the value \refe{MT}. of the symmetric matrix ${M}$. \cqfd
\par\bs\noindent
{\bf Proof of Proposition \ref{pro.cor}.}
Let $v=(v_1,v_2)\in W^{1,r}(\Om)^2$. Considering the functions $v_\ep^i$, $i=1,2$, which are defined by \refe{vpie}. and satisfy the convergences \refe{cvpie}., we have
\beq\label{Ee}
\ba{rl}
\dis E_\ep:= & \dis \int_\Om\left|Du_\ep-Du-v_1\,Dv_\ep^1-v_2\,Dv_\ep^2\right|^2 dx
\\ \ecart
= & \dis \int_\Om|Du_\ep|^2\,dx-\int_\Om|Du|^2\,dx+\int_\Om\left(v_1^2\,|Dv_\ep^1|^2+v_2^2\,|Dv_\ep^2|^2\right)dx
\\ \ecart
& \dis -\,2\int_\Om\left(v_1\,Du_\ep:Dv_\ep^1+v_2\,Du_\ep:Dv_\ep^2\right)dx+o(1).
\ea
\eeq
Putting $u_\ep$ in equation \refe{pSeoe}. and $u$ in equation \refe{Be}. we get
\beq\label{e6t2}
\ba{l}
\dis \int_\Om|Du_\ep|^2\,dx-\int_\Om|Du|^2\,dx=\langle f,u\rangle_{H^{-1}(\Om)^2,H^1_0(\Om)^2}-\int_\Om|Du|^2\,dx+o(1)
\\ \ecart
\dis =\int_\Om \Ga{u}\cdot u\,dx+o(1)={\ga\over 4\,(\ga^2+1)}\int_\Om |u|^2\,dx+o(1).
\ea
\eeq
Moreover, putting $V_\ep^i$ in estimate \refe{eDVie}. together with $V_\ep^i=e_i$ in $Q_{r_\ep}$, \refe{eVPie}., \refe{gie}., \refe{ga}., and using the $\ep Y$-periodicity of $Dv_\ep^i$, we get
\[
|Dv_\ep^i|^2\confai\lim_{\ep\to 0}\left({1\over\ep^2}\moy_Y|DV_\ep^i|^2\,dy\right)={\ga\over 4}\quad\mbox{weakly-$*$ in }\M(\bar{\Om}),
\]
hence since $v_i\in C(\bar{\Om})$,
\beq\label{e7t2}
\int_\Om\left(v_1^2\,|Dv_\ep^1|^2+v_2^2\,|Dv_\ep^2|^2\right)dx={\ga\over 4}\int_\Om|v|^2\,dx+o(1).
\eeq
Estimates \refe{Ee}., \refe{e6t2}. and \refe{e7t2}. thus imply that
\beq\label{e8t2}
E_\ep={\ga\over 4\,(\ga^2+1)}\int_\Om|u|^2\,dx+{\ga\over 4}\int_\Om|v|^2\,dx
-2\int_\Om\left(v_1\,Du_\ep:Dv_\ep^1+v_2\,Du_\ep:Dv_\ep^2\right)dx+o(1).
\eeq
On the other hand, applying the estimate \refe{eDvie}. with the function $v=v_i\,u_\ep$, $i=1,2$, and using the convergences \refe{cvpie}., \refe{gie}., \refe{ga}. and \refe{nucexS}., we obtain
\[
\ba{ll}
\dis \int_\Om Dv_\ep^i:Du_\ep\,v_i\,dx & \dis=\int_\Om Dv_\ep^i:D(v_i\,u_\ep)\,dx+o(1)
\\ \ecart
& \dis ={\ga_\ep^i\over\ep^2}\,e_i\cdot\left(\int_\Om{1_{\om_\ep}\over|Q_{r_\ep}|}\,u_\ep\,v_i\,dx
-\int_\Om{1_{Y\setminus Q}\over|Y\setminus Q|}\left({x\over\ep}\right)u_\ep\,v_i\,dx\right)+o(1)
\\ \ecart
& \dis ={\ga^2\over 4\,(\ga^2+1)}\int_\Om e_i\cdot(\ga\,I-J)\,u\,v_i\,dx-{\ga\over 4}\int_\Om u_i\,v_i\,dx+o(1).
\ea
\]
This combined with \refe{e8t2}. yields
\beq\label{e9t2}
\ba{ll}
E_\ep= & \dis {\ga\over 4\,(\ga^2+1)}\int_\Om|u|^2\,dx+{\ga\over 4}\int_\Om|v|^2\,dx
\\ \ecart
& \dis +\,{\ga\over 2\,(\ga^2+1)}\int_\Om u\cdot v\,dx+{\ga^2\over 2\,(\ga^2+1)}\int_\Om Ju\cdot v\,dx+o(1).
\ea
\eeq
Putting the function
\[
v:=-\,{1\over\ga^2+1}\left(I+\ga\,J\right)u
\]
in \refe{e9t2}. we get
\beq\label{e10t2}
E_\ep=\int_\Om\left|Du_\ep-Du-v_1\,Dv_\ep^1-v_2\,Dv_\ep^2\right|^2 dx\limo 0.
\eeq
Finally, since the sequences $v_\ep^i$ strongly converge to zero in $L^{2r\over r-2} (\Om)^2$ by \refe{cvpie}. and $u\in W^{1,r}(\Om)^2$, the strong convergence \refe{scue}. is a straightforward consequence of \refe{e10t2}.. \cqfd

%
%
\end{document}